\documentclass[11pt]{amsart}
\usepackage{graphicx}
\usepackage{tikz}
\usepackage{latexsym,color,amsmath,systeme, amsthm,amssymb,amscd,amsfonts}
\usepackage{graphicx}
\usepackage{todonotes}

\usepackage{hyperref}
\usepackage{pgf,tikz,pgfplots, pgfplotstable}
\pgfplotsset{compat=1.14}
\usepackage{float}
\usepackage{mathrsfs}
\usetikzlibrary{arrows}
\usepackage[style=english]{csquotes}
\usepackage{enumerate}
\usepackage{lscape}
\usepackage[width=\textwidth]{caption}
\usepackage{subcaption}
\usepackage{amsmath}
\usepackage{amsfonts}
\usepackage{amssymb}
\usepackage{pgf,tikz}
\usepackage{tkz-tab}
\usepackage{pgfplots}
\usepackage{siunitx}
\usetikzlibrary{angles, quotes}
\usepackage{tikz-3dplot}
\usetikzlibrary{positioning, arrows.meta, calc}
\usepackage{amsmath, amssymb, amsfonts}
\usepackage{amssymb,amsthm,bbm,listings,todonotes,url,xcolor}
\usetikzlibrary{decorations.pathreplacing,matrix}

\usetikzlibrary{decorations.pathmorphing}
\tikzset{discont/.style={decoration={zigzag,segment length=12pt, amplitude=4pt},decorate}}
\pgfplotsset{compat=1.11}
\usetikzlibrary{intersections,fillbetween}
 \usetikzlibrary{decorations.text}
\usetikzlibrary{shapes,snakes,arrows,intersections, backgrounds}
\usetikzlibrary{scopes,svg.path,shapes.geometric,shadows}
\DeclareCaptionSubType*[arabic]{figure}
\tikzset{% was glyellow
	v1/.style={line width=.5pt,blue!33!black},
	v2/.style={line width=.5pt,blue!66!black},
	v3/.style={line width=.5pt,blue!33},
	v4/.style={line width=.5pt,blue!66},
	v5/.style={line width=.5pt,black}
}%https://sharelatex.tum.de/7368283976cwbyxdkytxyd

\definecolor{dred}{HTML}{C11B17}
\definecolor{dgreen}{HTML}{41A317}
\definecolor{dblue}{HTML}{00008B}
\definecolor{niceblue}{HTML}{00008B}
\definecolor{brilliantrose}{rgb}{1.0, 0.33, 0.64}
\definecolor{gold}{HTML}{ffd700}
\definecolor{lgold}{HTML}{ffe140}

\newcommand{\R}{\mathbb R}
\newcommand{\K}{\mathcal K}
\newcommand{\F}{\mathcal F}
\newcommand{\cV}{\mathcal V}
\newcommand{\VP}{\mathcal{VP}}
\newcommand{\cK}{\mathcal K}
\newcommand{\Z}{\mathcal Z}
\newcommand{\N}{\mathbb N}

\newcommand{\Pro}{\mathbb P}

\newcommand{\conv}{\mathrm{conv}}

\def\vol{{\rm vol}}
\def\Vol{{\rm Vol}}

\definecolor{zzttqq}{rgb}{0.6,0.2,0}
			\definecolor{ccqqqq}{rgb}{0.8,0,0}
\definecolor{ffvvqq}{rgb}{1,0.3333333333333333,0}
\definecolor{zzffqq}{rgb}{0.6,1,0}
\definecolor{qqwuqq}{rgb}{0,0.39215686274509803,0}
\definecolor{ffzzqq}{rgb}{1,0.6,0}
\definecolor{ffqqqq}{rgb}{1,0,0}
\definecolor{zzttqq}{rgb}{0.6,0.2,0}
\definecolor{uuuuuu}{rgb}{0.26666666666666666,0.26666666666666666,0.26666666666666666}

\newtheorem{thm}{Theorem}[section]
\newtheorem{lemma}[thm]{Lemma}

\newtheorem{cor}[thm]{Corollary}

\theoremstyle{definition}

\theoremstyle{remark}
\newtheorem{rmk}[thm]{Remark}

\newcommand{\Det}{\operatorname{Det}}

\newcommand{\Gr}{\operatorname{Gr}}

\newcommand{\onevec}{\mathbbm{1}}

\newtheorem{rem}{Remark}

\date{\today} 

\definecolor{gcolor}{RGB}{200,255,200} % Gena - green
\definecolor{kcolor}{RGB}{245,225,215} % Katja - khaki rose 
\definecolor{vcolor}{RGB}{240,230,255} % Vanya - violet 

\usetikzlibrary{calc,intersections}
\begin{document}
%%%%%%%%%%%%%%%%%%%%%%%%%%%%%%%%%%%%%%%%%%%%%

%\title{Mixed Volume Inequalities for Zonoids}

\title[Mixed Volume Inequalities for Zonoids]{On the Log-submodularity for zonoids: \\ from Mixed Volume inequalities to the Hypercube}



\author[G. Averkov]{Gennadiy Averkov}
\address{Fakult\"at 1, BTU Cottbus-Senftenberg, Platz der Deutschen Einheit 1, 03046 Cottbus, Germany}
\email{averkov@b-tu.de}

\author[K. von Dichter]{Katherina von Dichter}
\email{vondicht@b-tu.de}

\author[I. Soprunov]{Ivan Soprunov}
\address{Department of Mathematics and Statistics, Cleveland State University,  2121 Euclid Ave, Cleveland, Ohio, 44115 USA}
\email{i.soprunov@csuohio.edu}

\date{\today}\maketitle

\begin{abstract}
%We study the log-submodularity property of volumes of coordinate projections of zonoids and formulate a local Loomis-Whitney-type inequality. This inequality admits an equivalent formulation in terms of mixed volumes, and ultimately provides information about volume polynomial diagrams of zonoids. We reduce the conjectured inequality to a polynomial inequality whose variables are associated with the vertices of a hypercube and whose coefficients encode the volumes of 0/1 simplices. This reduction reveals unexpected connections between mixed volumes of zonoids, matroid theory, and real algebraic geometry.

%Convex geometry features 
%We introduce a framework of log-submodularity type inequalities for volumes of coordinate projections of convex bodies, extending the notion of submodular functions. This framework includes the Loomis-Whitney inequality and several of its generalizations.
In convex geometry, a number of inequalities for the volume of coordinate projections of convex bodies share a common structural property.
%are structured in some special way. 
We formalize these 
%structured inequalities 
as inequalities %being 
of log-submodularity type, providing a definition that stems from the notion of a submodular function.  The Loomis-Whitney inequality and some of its generalizations are the most prominent examples. 
Some inequalities of log-submodularity type are valid for zonoids but not for general convex bodies. They have a natural meaning in statistics and information theory.  
We propose a new inequality of log-submodularity type for zonoids. 
In dimension $3$, our inequality is equivalent to the log-submodularity conjecture for zonoids formulated by Fradelizi, Madiman, Meyer, and Zvavitch, 
which is true in dimension $3$ and false in higher dimensions. In contrast, our inequality is true in every dimension. 
We reduce our inequality to copositivity of a special polynomial, which we call the hypercube polynomial. 
 For 4-dimensional zonoids, we verify the copositivity of the respective hypercube polynomial algebraically by providing a sum-of-squares certificate. In every dimension, a proof of copositivity of the hypercube polynomial that uses probability theory and  well-known matrix inequalities has been produced by LLMs and brought to our attention. We present a detailed and easy-to-read version of this proof. 
\end{abstract}

\vspace{3mm}

\section{Introduction}\label{sec:IntrNotMain}

\subsection{On the log-submodularity property for zonoids}

A function $f : 2^{[n]} \to \R$ is called submodular if 
\[
f(I \cup J) + f(I \cap J) \le f(I) + f(J)
\]
%$ 
holds for all $I,J \subseteq [n]$. We study related properties of  the logarithmic volume of the coordinate projection of a convex body $K$ in $\R^n$, which is the function $f(I) = \log \vol P_I(K)$, with $P_I$ denoting the projection onto coordinates indexed by $I \subseteq [n]$. While there exist convex bodies $K$ of dimension $n \ge 3$, for which $f$ is not submodular ~\cite{MNZ}, it is known that $f$ satisfies certain submodularity-type inequalities. A prominent example is the classical Loomis-Whitney~\cite{LW}, which bounds the volume of a body by its projection volumes:
\[
\vol(K)^{n-1} \le \prod_{i=1}^{n} \vol(P_{e_i^\perp} K),
\]
for any $K\in\cK^n$. Here $P_{e_i^\perp}K$ denotes the projection of $K$ onto the coordinate hyperplane orthogonal to $e_i$. By taking the logarithm and using the notation $f(I) = \log \vol P_I(K)$, we can write the Loomis-Whitney inequality as 
\begin{equation}\label{eq:L-W}
(n-1)	f(\{1,\ldots,n\}) + f(\emptyset) \le \sum_{i=1}^n f([n] \setminus \{i\}). 
\end{equation}
This is an example of what we call a {submodularity-type inequality}, which we define as follows. 
For each subset $S\subseteq[n]$, let ${\bf 1}_S$ denote its characteristic vector. Consider
an affine combination
\[
\sum_{i=1}^k w_i {\bf 1}_{S_i} =
\sum_{i=1}^l u_i {\bf 1}_{T_i},\quad\text{where }\  \sum_{i=1}^k w_i =\sum_{i=1}^l u_i\ \text{ and } \ w_i,u_i>0.
\]
It defines the inequality 
\[
\sum_{i=1}^k w_i f(S_i) \leq
\sum_{i=1}^l u_i f(T_i),
\]
called a \emph{submodularity-type inequality}.

%Consider an inequality
%\[
%w_1 f(S_1)+\cdots+w_k f(S_k)
%\;\le\;
%u_1 f(T_1)+\cdots+u_\ell f(T_\ell),
%\]
%where \(w_1,\ldots,w_k,u_1,\ldots,u_\ell\ge 0\) and \(S_1,\ldots,S_k,T_1,\ldots,T_\ell\subseteq[n]\). For an index \(i\in[n]\), define its weight on the left-hand side by
%\[
%\sum_{\substack{j\in[k]\\ i\in S_j}} w_j,
%\]
%and its weight on the right-hand side is defined analogously. We call the inequality a \emph{submodularity-type inequality} if the weight of every index $i\in[n]$ is the same on both sides.

%This is an example of what we call a \emph{submodularity-type inequality}, which we define as follows. Consider an expression of the form $w_1 f(S_1) + \cdots + w_k f(S_k)$, where $w_1,\ldots,w_k \ge 0$ and $S_1,\ldots,S_k \subseteq [n]$. In this expression, we define the weight of an index $i \in [n]$ as the sum of all $w_j$ for all $j \in [k]$ that satisfy $i \in S_j$. We call an inequality  $w_1 f(S_1) + \cdots + w_k f(S_k) \le u_1 f(T_1) + \cdots + u_\ell f(T_\ell)$ with nonnegative coefficients $w_1,\ldots,w_k, u_1,\ldots, u_\ell \ge 0$ a submodularity-type inequality if, for each $i \in [n]$, the weight of $i$ with respect to the left-hand side coincides with the weight of $i$ with respect to the right-hand side of this inequality. 

%For example, in the Loomis-Whitney inequality the weight of each $i \in [n]$ is exactly $n-1$ on both sides of the inequality. 
Clearly, the inequality $f(I \cup J) + f(I \cap J) \le f(I)+f(J)$ in the definition of a submodular function is of submodularity type. The $s$-cover version of the Loomis-Whitney inequality by Bollobás and Thomason~\cite{BT} is also of submodularity type. Further inequalities for volumes of coordinate projections of convex bodies can be found in~\cite{ABBC,MNZ,M,L,AAGJMR,AGM}.

%\hrulefill 

%Extensions of this inequality were obtained by Bollobás and Thomason~\cite{BT}. Also, local $s$-cover versions were later established in~\cite{ABBC, MNZ}. Dually, Meyer~\cite{M} proved a lower bound via sections, generalized by Liakopoulos~\cite{L} and Alonso-Gutiérrez et al.~\cite{ABBC, AAGJMR}. Recently, a systematic treatment of these local inequalities and their functional counterparts was given by Alías, González Merino, and Marín Gimeno~\cite{AGM}.

%\hrulefill 

While for general convex bodies $f(I) = \log \vol(P_I K)$ fails to be submodular, one may investigate if it holds for more specific interesting  families of convex bodies. For example, submodularity holds for the family of parallelotopes. Fradelizi, Madiman, Meyer and Zvavitch \cite{FMMZ}  conjectured that the functions $f(I) = \log \vol(P_I K)$ is submodular when $K$ is a zonoid. Recall that a convex body is called a {\it zonotope} if it is the Minkowski sum of a finite number of line segments, called the generators of the zonotope. Limits of zonotopes in the Hausdorff metric are called {\it zonoids}. Note that parallelotopes, for which the conjecture is true, are just zonotopes in $\R^n$ that have $n$ linearly independent generators.

The log-submodularity conjecture has a natural information-theoretic meaning in statistics via the notion of Vitale zonoids~\cite{V}, an object capturing interrelations of a system of $n$ random variables ~\cite{DCT,Md}. Beyond the statistics, the submodularity phenomena have been observed in other areas of mathematics, in particular, in combinatorics and algebraic geometry~\cite{AH,BH}. %,  and convex geometry. 
Currently, experts try to identify hidden connections between objects from different areas of mathematics that possess the submodularity property. Thus, %our hope is 
we believe that understanding properties of the function $f(I) =  \log \vol(P_I A)$ will have an impact outside of convexity theory.

%would not only have an impact %inside convexity 
%within convex geometry but also in other related areas, such as probability theory~\cite{BBLM}, information theory ~\cite{CC, DCT, Md}, toric geometry~\cite{Ful, CLS}, algebraic combinatorics~\cite{BH, AH, Mur, Bjo} etc.

The first non-trivial case of the log-submodularity conjecture is the case of dimension $n=3$. It was proved in the original paper \cite{FMMZ} and, with an algebraic-combinatorial approach, in \cite{AS}.  
Very recently, the submodularity conjecture was disproved in any dimension $n \geq 4$ for $n$-dimensional zonotopes $A$ with at least $m = n+2$ generators, independently in \cite{FHMNZ} and \cite{Sko}.
This indicates the need for other inequalities that relate volumes of the coordinate projections of $A$. While the log-submodularity inequality is not true in general, there is still a chance that certain inequalities of submodularity type are satisfied. It is clear that if a certain inequality holds for every zonotope, then it also holds for all zonoids (by taking the limit). So, one can consider a zonotope $A$ with a fixed but arbitrary number $m \ge n$ of generators and then try to derive the inequality in that discrete setting.  This was the approach in \cite{FMMZ} and in \cite{AS} for the case of $n=3$. In particular, in \cite{AS} it was demonstrated that for $n=3$, the case of an arbitrary $m \ge 3$ can be reduced to the case $m=4$. Note also that in \cite{FMMZ} the  log-submodularity conjecture was verified for $m=n$ and an arbitrary $n \geq 3$. Thus, it is natural to have the number of generators $m$ of a zonotope $A$ as an additional parameter along with the dimension $n$.

For a zonoid $K$, we propose the following log-submodularity-type inequality for $f(I) = \log \vol P_I(K)$:
\begin{equation}\label{eq:partial_polymatroid}
f(\{1,\dots,n\})+(n-2)f(\{n\})
\le f(\{1,n\})+f(\{2,n\})+\cdots+f(\{n-1,n\}).
\end{equation}

Our choice of index sets chosen in \eqref{eq:partial_polymatroid} is motivated as follows. For each set $S$, let ${\bf 1}_S$ denote its incidence vector. For the family $\{\{1,\dots,n\},\{n\},\{1,n\},\dots,\{n-1,n\}\}$ we have $n+1$ vectors in $\mathbb R^n$, which are minimally dependent. They are related by one linear relation, up to scaling, is
\[
{\bf 1}_{\{1,\dots,n\}}+(n-2){\bf 1}_{\{n\}}
=
\sum_{i=1}^{n-1}{\bf 1}_{\{i,n\}}.
\]
Thus \eqref{eq:partial_polymatroid} is associated with this circuit. When $n=3$, it reduces to the original log-submodularity property. It is therefore natural to ask which other index-set families, or equivalently, which other circuits, admit the analogous inequality for $f(I)=\log\vol P_I(A)$, and whether our circuit is the only one forced by the geometry of projections.

\begin{figure}[htbp]
\centering
\begin{subfigure}[b]{0.3\textwidth}
\centering
\begin{tikzpicture}[scale=0.85, every node/.style={font=\footnotesize}]
\coordinate (A) at (0,2);
\coordinate (B) at (2,2);
\coordinate (C) at (2,0);
\coordinate (D) at (0,0);

\draw (A) -- (B) -- (C) -- (D) -- cycle;
\draw[dred, draw opacity=0.5] (B) -- (D);
\draw[dred, line width=1.6pt] (A) -- (C);

\fill (A) circle (1.8pt);
\fill (B) circle (1.8pt);
\fill (C) circle (1.8pt);
\draw[fill=white] (D) circle (1.8pt);

\node[above left]  at (A) {$\{1\}$};
\node[above right] at (B) {$\{1,2\}$};
\node[below right] at (C) {$\{2\}$};
\node[below left]  at (D) {$\emptyset$};
\end{tikzpicture}
\caption{Inequality (\ref{eq:L-W}) for $n=2$.}
\end{subfigure}
\hfill
\begin{subfigure}[b]{0.30\textwidth}
\centering
\begin{tikzpicture}[scale=0.85, every node/.style={font=\footnotesize}]
\coordinate (A) at (0,2);
\coordinate (B) at (2,2);
\coordinate (C) at (2,0);
\coordinate (D) at (0,0);

\draw (A) -- (B) -- (C) -- (D) -- cycle;
\draw[dred, draw opacity=0.5] (B) -- (D);
\draw[dred, line width=1.6pt] (A) -- (C);

\fill (A) circle (1.8pt);
\fill (B) circle (1.8pt);
\fill (C) circle (1.8pt);
\draw[fill=white] (D) circle (1.8pt);

\node[above left]  at (A) {$\{1,3\}$};
\node[above right] at (B) {$\{1,2,3\}$};
\node[below right] at (C) {$\{2,3\}$};
\node[below left]  at (D) {$\{3\}$};
\end{tikzpicture}
\caption{Inequality (\ref{eq:partial_polymatroid}) for $n=3$.}
\end{subfigure}
\hfill
\begin{subfigure}[b]{0.34\textwidth}
\centering
\begin{tikzpicture}[scale=1.5, x={(1.03,0.034)}, y={(0,1)}, z={(-0.30,-0.38)},
                    every node/.style={font=\footnotesize}]

\coordinate (s2) at (0,0,0);
\coordinate (s4) at (1,0,0);
\coordinate (s8) at (1,1,0);
\coordinate (s6) at (0,1,0);

\coordinate (s1) at (0,0,1);
\coordinate (s3) at (1,0,1);
\coordinate (s7) at (1,1,1);
\coordinate (s5) at (0,1,1);

\draw[dashed] (s1) -- (s2) -- (s4);
\draw (s4) -- (s3) -- (s1);
\draw (s5) -- (s6) -- (s8) -- (s7) -- cycle;
\draw (s1) -- (s5);
\draw[dashed] (s2) -- (s6);
\draw (s3) -- (s7);
\draw (s4) -- (s8);

\draw[dred, thick] (s2) -- (0.96,0,0.96)  -- (0,0.96,0.96) -- cycle;
\fill[dred, opacity=0.15] (s2) -- (s3) -- (s5) -- cycle;
\draw[dred, draw opacity=0.5] (s1) -- (s8);

\fill[black] (s1) circle (0.035);
\fill[black] (s2) circle (0.035);
\fill[black] (s3) circle (0.035);
\fill[black] (s5) circle (0.035);
\fill[black] (s8) circle (0.035);

\coordinate (P) at (0.3333, 0.3333, 0.6667);
\fill[dred] (P) circle (0.035);

\node at (s1) [below left]  {$\{4\}$};
\node at (s2) [below right] {$\{2,4\}$};
\node at (s3) [right]       {$\{1,4\}$};
\node at (s5) [below left]  {$\{3,4\}$};
\node at (s8) [above right] {$\{1,2,3,4\}$};
\end{tikzpicture}
\caption{Inequality (\ref{eq:partial_polymatroid}) for $n=4$.}
\end{subfigure}
\caption{The circuits underlying some of submodularity-type inequalities.}
\end{figure}

Our contributions are the following: 

\begin{enumerate}[(i)] 
    \item We show that the log-submodularity conjecture is true for $n$-dimensional zonotopes $A$ with $m$ generators if and only if $m \le n+1$. The positive answer to the conjecture in the case of $m=n+1$ is a direct consequence of a general duality principle of diagrams of zonotopes (see Corollary \ref{cor:d+1}). This result was also established independently in \cite{FHMNZ}.
    \item We prove a ``Helly-type'' result that the conjecture about $f(I) = \log \vol P_I(A)$ being a partial polymatroid (see \eqref{eq:partial_polymatroid}) can be reduced from the case of an arbitrary $m \ge n$ to the case $m=2^{n-1}$. See Section~\ref{sec:HIvsPI}. 
The inequality \eqref{eq:partial_polymatroid}  admits a natural interpretation as a relation between points on the absolute Grassmannian $\operatorname{Abs}(\operatorname{Gr}(n,m+n-1))$. The reduction therefore allows one to reduce the verification of this relation on 
 the fixed, independent of $m$, absolute Grassmannian $\operatorname{Abs}(\operatorname{Gr}(n,2^{n-1}+n-1))$, see Remark~\ref{rem:grassmannian}.

	\item 
    We show that the conjecture \eqref{eq:partial_polymatroid} is equivalent to the following purely combinatorial statement for $d = n-1$, which we call the \emph{Hypercube inequality}:
\begin{equation}\label{eq:hypercube_ineq}
	\left( 
        \sum_{S \in \binom{\{0,1\}^d}{d+1}} 
            \Vol(S)\prod_{v \in S} x_v
            \right)
        \left( 
        \sum_{v \in \{0,1\}^d} x_v 
    \right)^{d-1}  
    \le 
    \prod_{\substack{F\subset [0,1]^d\\ F\, \text{facet}}} 
    \sum_{v \in F \cap \{0,1\}^d} 
        x_v,
\end{equation}
where we assign a non‑negative variable $x_v \in \mathbb{R}_+$ to each vertex $v$ of
the hypercube $[0,1]^d$ and $\Vol(S)=d!\vol(\conv(S))$ denotes the normalized $d$-dimensional volume of the simplex with vertex set $S$. The product in the right-hand side runs over all $2d$ facets of the cube: each factor corresponds to the sum of $x_v$ over the vertices of that facet. See Section~\ref{sec:HIvsPI}. 
%\KSays{The vertices of the $d$-cube correspond naturally to Plücker coordinates on $\operatorname{Abs}(\operatorname{Gr}(n,2^{n-1}+n-1))$, and the coefficients $\Vol(S)$ appearing in the hypercube inequality are precisely the absolute values of those coordinates. Thus, the hypercube inequality becomes a statement about these Plücker coordinates, and our geometric inequalities for zonoids translate into purely combinatorial statements of independent interest in matroid theory, $M$-convex sets~\cite{Mur}, and the theory of strongly log‑concave polynomials~\cite{BH}. }

    \item We show that the Hypercube inequality \eqref{eq:hypercube_ineq} is true for $d=3$ and provide a sum-of-squares certificate (and a non-negative part), thereby linking it to real algebraic geometry and the theory of non-negative polynomials~\cite{Re,Ma}. We also describe its equality cases and show that it is tight. See Section \ref{sec:proof_HI}. 
        
    \item After a previous version of this manuscript was made available on arxiv.org, we received an LLM-generated proof of the general form of the Hypercube inequality \eqref{eq:hypercube_ineq} for verification %(Dmitry Melyushchenkov, private communication).
    \footnote{Dmitry Melyushchenkov, private communication.} Later the same proof was again found by an LLM and was formally verified in Lean \cite{ZhMa}. %This was communicated to us by Wenlin Zhang and Haobo Ma.
    \footnote{Wenlin Zhang and Haobo Ma, private communication.} We checked the correctness of the proof and include it in Section~\ref{sect:probabilistic:proof}, with the details that facilitate understanding the idea of the proof. The focus of our proof for $d=3$ was on the derivation of an algebraic sum-of-squares certificate, as we are also interested in connections between real algebra and convexity. In contrast, the LLM-generated proof adopts a probabilistic viewpoint and proceeds by combining a sequence of estimates and some rather well-known identities and inequalities for matrices. So, in this manuscript, we present both proofs: our proof for $d=3$ as well as the LLM-generated proof for arbitrary $d$, because we find both proofs valuable.
    
\end{enumerate}

From the perspective of real algebra, it would be interesting to find a sum-of-square based proof for the Hypercube inequality in an arbitrary dimension $d$. 
Note that it is practically impossible to write the Hypercube inequality \eqref{eq:hypercube_ineq} explicitly: on the left-hand side we would need a complete description of all simplices spanned by vertices of the hypercube together with their volumes, which, in general, is unknown. Even the problem of finding the maximum-volume simplex in a hypercube is itself a classical open question, closely related to the Hadamard maximum determinant problem \cite{HKL, NWZ}. Note that in the next dimension, $d=4$, the expansion of \eqref{eq:hypercube_ineq} consists of roughly $400000$ terms.

%Recall that a convex body is a {\it zonotope} if it is the Minkowski sum of a finite number of line segments, called {\it generators} of the zonotope. Limits of zonotopes in the Hausdorff metric are called {\it zonoids}. We denote the family of all convex bodies in $\R^n$ by $\K^n$ and the family of all zonoids in $\mathbb{R}^n$ by $\Z^n$. Furthermore, we use $\Z^n_m$ to denote the family of all zonotopes in $\R^n$ generated by at most $m$ segments. 

%In the more restricted case of zonoids, it was conjectured by Fradelizi, Madiman, Meyer, and Zvavitch~\cite{FMMZ} that there are stronger structural inequalities valid \KSays{for the diagram $\VP(\Z^n, \{[0,e_1]\},\dots, \{[0,e_n]\})$, the so called {\it log-submodularity} property}.

%Recall that $f : 2^{ [n]} \to \R$ is called {\it submodular} if the inequality 
%\[
%	f(I \cap J) + f(I \cap J) \le f(I) + f(J) 
%\]
%holds for all $I, J \subseteq [n]$. Consider an $n$-dimensional zonoid $A$  and, for every subset $I \subseteq [n]:=\{1,\ldots,n\}$, let $P_I$ be the projection onto the coordinate subspace indexed by $I$. The log-submodularity conjecture from~\cite{FMMZ}  claims that the map $f(I) =  \log \vol(P_I A)$ is a submodular function. 

\subsection{Web of diagrams}
Diagrams aim to provide an exhaustive description of relations between a given set of functionals, with the relations usually formulated in terms of inequalities. The problem of describing diagrams for typical functionals in convex geometry, such as volume, surface area, width, diameter, inradius, circumradius, and perimeter, was raised by Santal\'o \cite{Santalo61}, based on the earlier ideas of Blaschke~\cite{Bl}. 
%The study of Blaschke-Santal\'o diagrams remains a highly active area of research to this day, with numerous open problems still unresolved (see, e.g.,~\cite{SY, HC}). Recently there have been different contributions to the understanding of these diagrams, both theoretical and computational (see, e.g.,~\cite{BG, BGM, FLP, Fto, FHL, FMZ, SZ}). % BGR1, BR, BGR2, 
 %Even before Santal\'o, Heine \cite{H} had suggested to study the diagrams for the system of mixed volumes of convex bodies. Shephard \cite{Sh} generalized Heine’s determinantal inequalities to arbitrary dimensions and revealed that the classical Minkowski, Aleksandrov-Fenchel and determinantal inequalities are in general not enough to describe all possible relations between mixed volumes, even for planar convex bodies. 
 Heine \cite{H} in the planar case and Shephard \cite{Sh} in higher-dimensional cases initiated the study of diagrams
 of mixed volumes of convex bodies. This laid the foundation for the Heine-Shephard problem, which asks for a complete system of all relations between mixed volumes of $k$ convex bodies in $\R^n$, for given $k$ and $n$.

Equivalently, this can be seen as the problem of characterizing families of volume polynomials, a problem that has attracted attention in the algebraic combinatorics community since the seminal work of Brändén and Huh on Lorentzian polynomials \cite{BH}, where they, in particularly, have shown that Lorentzian polynomials form a strictly larger class than volume polynomials of arbitrary convex bodies. In the present paper, we do a similar comparison, contrasting volume polynomials of general convex bodies with those of zonoids. Thus, our focus lies on the study of diagrams for the mixed volumes of systems of zonoids and, sometimes more specifically, of zonotopes.

We introduce notation and terminology to deal with different types of diagrams. Given $\ell$ families $\F_1,\dots, \F_\ell$ of convex bodies in $\R^n$, define the {\it volume polynomial diagram}
$$
\VP(\F_1,\dots, \F_\ell)=\{\vol(x_1K_1+\dots+x_\ell K_\ell) \ : \ K_i\in\F_i,\ 1\leq i\leq \ell\},
$$
where $\vol$ denotes the Euclidean volume in $\R^n$ and $x_1,\dots, x_\ell$ are nonnegative scaling factors.
By a result of Minkowski, $\vol(x_1K_1+\dots+x_\ell K_\ell)$ is a polynomial in the nonnegative variables $x_1,\ldots,x_\ell$ and its coefficients are all possible mixed volumes involving $K_1,\dots, K_\ell$, see Section~\ref{sec:def}. 

%The coefficients of the volume polynomial $\vol(x_1K_1+\dots+x_\ell K_\ell)$
%are all possible mixed volumes involving $K_1,\dots, K_\ell$, see Section~\ref{sec:def}. 

Let us fix a family $\F$ of convex bodies in $\R^n$. It defines three diagrams $\VP(\underbrace{\F,\dots, \F}_{\ell})$, 
$\VP(\F,\underbrace{\Z_1^n,\dots, \Z_1^n}_{k})$, and $\VP(\F, \{[0,e_1]\},\dots, \{[0,e_n]\})$, where $[0,e_i]$ denotes the unit line segment 
generated by the $i$-th standard basis vector $e_i$.
%, where by abuse of notation we write $[0,e_i]$ instead of the one-element family $\{[0,e_i]\}$. 
Describing the first diagram corresponds to the Heine-Shephard problem for the family $\F$, whereas describing the
third diagram corresponds to relating the volumes of all coordinate projections for bodies in $\F$. The second 
diagram is a coordinate-free analog of the third diagram. In particular, for $k=n$, the second diagram contains the third one. The above three diagrams 
can be considered for the following natural choices of $\F$: all convex bodies $\cK^n$, all zonoids $\Z^n$, and
zonotopes with a bounded number of generators $\Z^n_m$. Our focus in this paper is on $\F=\Z^n$ and $\F=\Z^n_m$. Observe that the diagrams for $\F=\Z^n$ can be obtained as a limit case of the corresponding diagrams for $\F=\Z^n_m$, as $m\to\infty$.

%TODO: start introd with Loomis-Whitney. 

Observe that the results Loomis--Whitney type inequalities provide partial description for the general diagrams $\VP(\K^n,\underbrace{\Z_1^n,\dots, \Z_1^n}_{k})$ and $\VP(\K^n, \{[0,e_1]\},\dots, \{[0,e_n]\})$.

%Observe that the log-submodularity type together with the original log-submodularity property from~\cite{FMMZ} are local Loomis--Whitney type inequalities. The classical Loomis--Whitney inequality~\cite{LW} bounds the volume of a body by its projection volumes:
%Projection inequalities bounding the volume of a body by its projection volumes date back to 
%\[
%\vol(K)^{n-1} \le \prod_{i=1}^{n} \vol(P_{e_i^\perp} K),
%\]
%for any $K\in\cK^n$. Here $P_{e_i^\perp}K$ denotes the projection of $K$ onto the coordinate hyperplane orthogonal to $e_i$. Extensions of this inequality were obtained by Bollobás and Thomason~\cite{BT}. Also, local $s$-cover versions were later established in~\cite{ABBC, MNZ}. Dually, Meyer~\cite{M} proved a lower bound via sections, generalized by Liakopoulos~\cite{L} and Alonso-Gutiérrez et al.~\cite{ABBC, AAGJMR}. Recently, a systematic treatment of these local inequalities and their functional counterparts was given by Alías, González Merino, and Marín Gimeno~\cite{AGM}. All these results provide partial description for the general diagrams $\VP(\K^n,\underbrace{\Z_1^n,\dots, \Z_1^n}_{k})$ and $\VP(\K^n, \{[0,e_1]\},\dots, \{[0,e_n]\})$.

%In the more restricted case of zonoids, it was conjectured by Fradelizi, Madiman, Meyer, and Zvavitch~\cite{FMMZ} that there are stronger structural inequalities valid for the diagram $\VP(\Z^n, \{[0,e_1]\},\dots, \{[0,e_n]\})$, the so called {\it log-submodularity} property.

%The contribution of this paper are the following: 

%\begin{enumerate}[(i)] 
%    \item 
%TODO: smoothen transition here. 

Regarding the case of zonotopes we have the following contribution. 

Let us abbreviate  the diagram $\VP(\Z_m^n,\underbrace{\Z_1^n,\dots, \Z_1^n}_{k})$ by $\cV(n,m,k)$.
    We establish a duality between the projection diagrams $\cV(n,m,k)$ and $\cV(m+k-n,m,k)$. This duality principle ultimately follows from the duality of the Grassmannians, a classical topic in algebraic geometry. See Section \ref{sec:duality}. This result was also established independently in \cite{FHMNZ}.

\section{Definitions and notations}
\label{sec:def}
Recall that a subset of $\mathbb{R}^n$ is called a \emph{convex body} if it is compact, convex, and non-empty.  For a convex body $K\subset\R^n$ we use $\vol(K)$ to denote its Euclidean $n$-dimensional volume.

The \emph{Hausdorff metric} of convex bodies $K$ and $L$ is the least $\rho \geq 0$ such that every point of $K$ is at distance at most $\rho$ from some point of $L$, and vice versa, every point of $L$ is at distance at most $\rho$ from some point of $K$. The \emph{Minkowski sum} of $K, L \subset \mathbb{R}^n$ is defined by $K + L = \{\, p + q : p \in K, \, q \in L \,\}$, and the non-negative scaling of $K \subset \mathbb{R}^n$ by $\lambda \in \mathbb{R}_+$ is defined by $\lambda K = \{\, \lambda p : p \in K \,\}$. For any $X\subset\R^n$ let $\conv(X)$ denote the \emph{convex hull} of $X$.  A \emph{zonotope} is a convex polytope which is the Minkowski sum of finitely many segments, %i.e., linear projections of cubes 
and a \emph{zonoid} is defined as the limit of zonotopes in the Hausdorff metric. We denote the family of all zonoids in $\mathbb{R}^n$ by $\Z^n$ and the family of zonotopes in $\mathbb{R}^n$ generated by at most $m$ segments by $\Z^n_m$.

For convex bodies $K_1,\dots,K_n\subset\mathbb{R}^n$, Minkowski's theorem %fundamental theorem of mixed volumes 
guarantees the existence of a unique symmetric, multilinear functional $V(K_1,\dots,K_n)$ (with respect to Minkowski addition) that satisfies $V(K,\dots,K)=\operatorname{vol}(K)$ for any convex body $K$ (see~\cite[Section~5.1]{Sch} for details). This functional, called the \emph{mixed volume}, can be expressed via the inclusion–exclusion formula
\[
V(K_1,\dots,K_n)=\frac{1}{n!}\sum_{p=1}^n(-1)^{n+p}\!\!
\sum_{1\le i_1<\dots<i_p\le n}
\operatorname{vol}\!\left(K_{i_1}+\cdots+K_{i_p}\right). 
\]

A direct consequence of Minkowski's theorem is that for any convex bodies $K_1,\dots,K_\ell\subset\mathbb{R}^n$ and non‑negative real numbers $x_1,\dots,x_\ell$, the volume of the Minkowski sum $x_1K_1+\cdots+x_\ell K_\ell$ is a homogeneous polynomial of degree $n$ in the $x_i$. This polynomial, known as the \emph{volume polynomial} of $K=(K_1,\dots,K_\ell)$, satisfies
\begin{align*}
f_K(x) &:=\operatorname{vol}\bigl(x_1K_1+\cdots+x_\ell K_\ell\bigr)\\
&= V\bigl(x_1K_1+\cdots+x_\ell K_\ell,\;\dots,\;x_1K_1+\cdots+x_\ell K_\ell\bigr)\\
&= \sum_{i_1,\dots,i_n=1}^\ell \,V(K_{i_1},\dots,K_{i_n})\,x_{i_1}\cdots x_{i_n}.
\end{align*}

Throughout the paper, whenever we consider inequalities for zonoids $A, B_1, \dots, B_k \subset \mathbb{R}^n$ which are homogeneous of degree 1 in each $B_i$, $i=1,\dots,k$,
%(and unless specified otherwise), 
we may assume without loss of generality that
\begin{equation}\label{eq:settings}
A = \sum_{i=1}^{m} [0, a_i], \qquad B_i = [0, b_i] \quad i=1,\dots,k,
\end{equation}
for some $m \geq n$ and vectors
$a_1,\dots, a_m, b_1,\dots,b_k$ in $\R^n$. %being an orthonormal system in $\mathbb{R}^n$. 

Indeed, every zonoid can be approximated by a sequence of zonotopes (finite Minkowski sums of segments) in the Hausdorff metric, and mixed volumes are continuous with respect to this metric. By a standard limit argument, if the inequality holds for all such approximating zonotopes, it holds for all zonoids.
Mixed volumes are translation invariant, hence, we may assume that every generator
of every zonotope contains the origin. Finally, mixed volumes are multilinear with respect to Minkowski addition, so if an inequality is homogeneous of degree 1 in a zonotope $B_i$, 
it suffices to verify it when $B_i$ is a segment.  

%Using the homogeneity of mixed volumes, we may also assume that $\{b_1,\dots,b_k\}$ is an orthonormal system.

%The all-ones vector is denoted by $\onevec=(1,\ldots,1)$ and we write $\onevec_n$ to indicate its dimension.

In this paper we often identify a zonotope $A = \sum_{i=1}^{m} [0, a_i]$ with the matrix
$A\in\R^{n\times m}$ with columns $a_1,\dots, a_m$. Let $I\subset [m]$ have size $n$ and
$A_I$ be the corresponding square submatrix of $A$. 
%We use $\Det(A_I) := |\det(A_I)|$ to denote the absolute value of the determinant. 
It is well known, and easy to verify, that for a zonotope $A$ in $\R^n$ one has
\begin{align} \label{vol:zon} 
\vol(A)  = \sum_{I \in \binom{[m]}{n}} |\det(A_I)|. %\Det(A_I) .
\end{align}

%We write $(a_1,\ldots,a_m)$ for the matrix
%\[
%A =
%\begin{pmatrix}
%| &  & | \\
%a_1 & \cdots & a_m \\
%| &  & |
%\end{pmatrix},
%\]
%whenever it is convenient to view $A$ as a system of its columns and define the zonotope generated by the columns of $A$ by 
%\[
%	Z(A) := Z(a_1,\ldots,a_m) = A [0,1]^n = [0,1] a_1 + \cdots + [0,1] a_m. 
%\]
%
%For a square matrix $A = (a_1,\ldots,a_n) \in \mathbb{R}^{n \times n}$, we write $\Det(A) := |\det(A)|$ for its absolute determinant. It is well known, and easy to verify, that for $A = (a_1,\ldots,a_m) \in \mathbb{R}^{n \times m}$,
%\begin{align} \label{vol:zon} 
%\vol(Z(A)) := \vol(Z(a_1,\ldots,a_m)) = \sum_{I \in \binom{[m]}{n}} \Det\big( (a_i)_{i \in I} \big).
%\end{align}

%Applying \eqref{vol:zon} to the vectors $x_i a_i$ with $x_i \in \mathbb{R}_+$, we obtain
%\begin{align*}
%\vol(Z(x_1 a_1,\ldots,x_m a_m)) 
%= \sum_{I \in \binom{[m]}{n}} \left( \prod_{i \in I} x_i \right) \Det\big( (a_i)_{i \in I} \big), 
%\end{align*}
%which is a volume polynomial for the system of the segments $[0,a_1],\ldots,[0,a_m]$ in $\R^n$. 

Let $A = \sum_{i=1}^{m} [0, a_i]\subset\R^n$ be a zonotope with $m$ generators and 
$B_i=[0,b_i]$ for $i=1,\dots, k$ be segments in $\R^n$. They determine the volume polynomial
$$f_{A,B}(x)=\vol(A+x_1B_1+\dots+x_kB_k).$$ 
This is a (nonhomogeneous) multi-affine polynomial
in $x_1,\dots, x_k$.  Generically, the coefficients of $f_{A,B}(x)$ are the volumes of projections of $A$ orthogonal to possible subspaces defined by $B_1,\dots, B_k$. Indeed, assume first that $\{ b_1,\dots, b_k \}$ %$(b_1,\dots, b_k)$ 
is an orthonormal system. Then, by the multilinearity of the mixed volume, we have
\begin{align}\label{e:vol-poly}
 f_{A,B}(x)= \sum_{p=0}^k\sum_{|J|=p}\frac{n!}{(n-p)!}V(A[n-p],B_J)\,x_{J}=\sum_{p=0}^k\sum_{|J|=p}\vol(P_{[B_{J}]^\perp}A)\,x_{J},
\end{align}
where for $J=\{i_1<\dots<i_p\}\subset[k]$ we use
$P_{[B_{J}]^\perp}$ to denote the projection operator onto the subspace orthogonal to the span
of $b_{i_1},\dots,b_{i_p}$,  and $x_{J}=x_{i_1}\cdots x_{i_p}$ is the corresponding monomial. 
When $b_1,\dots, b_k$ are linearly independent but not necessarily orthonormal, one can apply a linear transformation to make them orthonormal. In this case $f_{A,B}$ is rescaled by a positive constant. Finally, 
when $b_1,\dots, b_k$ are linearly dependent $f_{A,B}$ reduces to a volume polynomial with fewer variables
after a linear change of variables.

Note that for every $J\subset[k]$ the projection $P_{[B_{J}]^\perp}A$ is a zonotope and, hence, one can 
use (\ref{vol:zon}) to compute the coefficients of $f_{A,B}(x)$ using absolute determinants.

Define the {\it projection diagram}
$\cV(k,n,m)$ as the set of all volume polynomials
$$\cV(k,n,m) = \{\vol(A+x_1B_1+\dots+x_kB_k) \ : \ A\in\Z_m^n, B_i\in\Z_1^n\}.$$
Note that $\cV(k,n,m)$ is a nonhomogeneous version of the diagram $\VP(\Z_m^n,\underbrace{\Z_1^n,\dots, \Z_1^n}_{k})$  
defined in the introduction.

\section{Duality for projection diagrams of zonotopes}\label{sec:duality}

%\Katja{finish the part about diagrams (To do: Vanya). }

%\Katja{adapt the name of mv diagrams (To do: Vanya). }

In this section we establish a duality between the projection diagrams $\cV(n,m,k)$ and $\cV(m+k-n,m,k)$. It is based on the standard duality between the Grassmannians $\Gr(n,N)$ and $\Gr(N-n,N)$, which we recall next. See also \cite[Lecture 6]{Ha}.

Let $V$ be an $N$-dimensional vector space and $V^*$ its dual. With every $n$-plane $L\subset V$ one can associate its annihilator 
$L^\perp=\{u\in V^* : u(v)=0, \forall v\in L\}$ which is a $(N-n)$-plane in $V^*$. This defines a map $\varphi:\Gr(n,N)\longrightarrow\Gr(N-n,N)$, $L\mapsto L^\perp$. 
%Under the Plücker embedding, this map is induced by the natural isomorphism $\bigwedge^d V\to\left(\bigwedge^{n-d}V\right)^*$ which comes from the
%perfect pairing $\bigwedge^d V\times \bigwedge^{n-d}V\to\bigwedge^n V$, $(\alpha,\beta)\mapsto\alpha\wedge\beta=c\omega$ where $\omega$ is a choice of a volume formthat gives an identification $\bigwedge^n V\cong\R$. Now, 
Under the Plücker embedding this map is described as follows. Let $\{v_1,\dots, v_n\}$ be a basis for $L$
 and $p(L)=[v_1\wedge\dots\wedge v_n]\in\Pro(\bigwedge^n V)$ be the corresponding Plücker point. Extend $\{v_1,\dots, v_n\}$ to a basis $\{v_1,\dots, v_n,v_{n+1},\dots, v_N\}$ for $V$ and let $\{u_1,\dots, u_n,u_{n+1},\dots, u_N\}$ be the dual basis for $V^*$. Then
 $\{u_{n+1},\dots,u_N\}$ is a basis for $L^\perp$ and, hence,
 $\varphi$ sends $p(L)$ to $p(L^\perp)=[u_{n+1}\wedge\dots\wedge u_N]\in \Pro(\bigwedge^{N-n} V^*)$. Note that the resulting projective point is independent of the choice of the extended basis.

 Now, let $V=\R^N$ with the standard basis $e_1,\dots, e_N$. The standard dot product on $V$ provides the identification $V^*=\R^N$, with $e_1,\dots, e_N$ being the dual basis. Also, $L^\perp$ is identified with the orthogonal complement of $L$. The Plücker coordinates $p(L)$ are the maximal minors $p_I(L)$ of the matrix whose rows span $L$:
$$
v_1\wedge\cdots\wedge v_n=\sum_{|I|=n} p_I(L) e_I,
$$
where $e_I=e_{i_1}\wedge\dots\wedge e_{i_n}\in\bigwedge^n V$
for $I=\{i_1<\cdots<i_n\}\subset[N]$. 

Note that the image of $e_I$ under $\varphi$ is $\pm e_{I^c}$, where $I^c=[N]\setminus I$ is the complement of $I$ and the sign is determined by the permutation $(I,I^c)$. Therefore,
$$
\varphi(v_1\wedge\dots\wedge v_n)=\sum_{|I|=n} \pm p_I(L) e_{I^c},
$$
which represents $p(L^\perp)$ in the standard basis of $\bigwedge^{N-n} V$.
This shows that for every $I\subset[N]$ of size $n$ we have
\begin{equation}\label{e:Grass-dual}
p_{I^c}(L^\perp)=\pm p_I(L),   
\end{equation}
up to a scalar independent of $I$.

\begin{thm}
The map $f(x_1,\dots, x_k)\mapsto x_1\cdots x_k f(x_1^{-1},\dots, x_k^{-1})$ defines a bijection between the
mixed volume diagrams $\cV(n,m,k)$ and $\cV(m+k-n,m,k)$.
\end{thm}

\begin{proof} 
Let $f_{A,B}(x)=\vol_d(A+x_1B_1+\dots+x_kB_k)$ be the volume polynomial of $A\in\Z_m^n$ and $B_1,\dots, B_k\in\Z_1^n$. Without loss of generality, we may assume that $B_1,\dots, B_k$ are pairwise orthogonal unit segments.

Form a matrix $M\in\R^{n\times (m+k)}$ whose columns are the generators of $A$ and $B_1,\dots, B_k$ and let $L$ be the row span of $M$. Let $M^\perp\in\R^{(m+k-n)\times(m+k)}$ be a matrix whose row span is the orthogonal complement $L^\perp$ and let $A^*\in\Z_m^{m+k-n}$ be the zonotope generated by the first $m$ columns of $M^\perp$ and $B^*_1,\dots, B^*_k\in\Z_1^{m+k-n}$ be segments generated by the last $k$ columns of $M^\perp$. Recall that by (\ref{e:vol-poly})
$$
f_{A,B}(x)=\sum_{p=0}^k\sum_{|J|=p}\vol(P_{[B_{J}]^\perp}A)\,x_{J}.
$$
Note that the volume of the projection $P_{[B_{J}]^\perp}A$ equals the sum of the absolute values of the maximal minors that use exactly those of the last $k$ columns of $M$ that correspond to the subset $J\subset[k]$. Thus, we can write
$$
\vol(P_{[B_{J}]^\perp}A) = \sum_{S\subset I\subset T}|p_I(L)|,
$$ 
where the sum is over subsets $I\subset[m+k]$ of size $n$, $S=\{m+i : i\in J\}$, and $T=[m]\cup S$.

Applying the duality (\ref{e:Grass-dual}), we get (up to some positive scalar)
$$
\sum_{S\subset I\subset T}|p_I(L)| = \sum_{T^c\subset I^c\subset S^c}|p_{I^c}(L^\perp)|,
$$
where the complements are taken with respect to the set $[m+k]$. Note that $T^c=\{m+i : i\in [k]\setminus J\}$ and $S^c=[m]\cup T^c$.
Therefore, 
$$
\sum_{T^c\subset I^c\subset S^c}|p_{I^c}(L^\perp)|=\vol(P_{[B^*_{[k]\setminus J}]^\perp}A^*).
$$
Comparing this with the formula for the volume polynomial (\ref{e:vol-poly}) and noting $x_J=x_1\cdots x_k(x_{[k]\setminus J})^{-1}$ we obtain (up to a positive scalar)
$$f_{A,B}(x_1,\dots, x_k)=x_1\cdots x_k f_{A^*,B^*}(x_1^{-1},\dots, x_k^{-1}),$$
which completes the proof.
\end{proof}

The above theorem implies that any coefficient inequality that holds for elements of $\cV(n,m,k)$ can be transformed to the corresponding coefficient inequality for elements of $\cV(m+k-n,m,k)$. We apply this observation in the case of local Loomis--Whitney-type inequality.

\begin{cor}\label{C:dual-inequalities}
Let $k,n,m\in\N$ with $k\leq n\leq m$ and $c_{k,n,m}\in\R_{+}$.
Then the inequality
\begin{equation}\label{eq:projection-2}
\vol(A)^{k-1} \, \vol\bigl( P_{[b_1,\dots,b_k]^\perp} A \bigr)
\;\leq\; c_{k,n,m}
\prod_{\ell=1}^k \vol\bigl( P_{b_\ell^\perp} A \bigr),
\end{equation}
holds for any $A\in\Z_m^n$ and any orthonormal set $\{b_1,\dots,b_k\}$ in $\R^n$ if and only if the inequality 
\begin{equation}\label{eq:projection_dual-2}
\vol(A) \, \vol\bigl( P_{[b_1,\dots, b_k]^\perp} A \bigr)^{k-1}
\;\le\; c_{k,n,m}
\prod_{\ell=1}^k \vol\bigl( P_{[b_1 \dots, \hat b_\ell,\dots, b_k]^\perp} A \bigr)
\end{equation}
holds for any $A\in\Z_m^{m+k-n}$ and any orthonormal set $\{b_1,\dots,b_{k}\}$ in $\R^{m+k-n}$.
\end{cor}

The particular case $m=n+1$ and $k=2$ provides a proof of the log-submodularity property for zonotopes with $n+1$ generators.

\begin{cor}\label{cor:d+1}
    Let $A\subset\R^n$ be a zonotope with $n+1$ generators. Then for any zonotopes $B_1,B_2\subset\R^n$ the following inequality holds
    \begin{equation}\label{e:d+1}
    \vol(A)V(A[n-2],B_1,B_2)\leq V(A[n-1],B_1)V(A[n-1],B_2).   
    \end{equation}
\end{cor}

\begin{proof}
By multilinearity of the mixed volume, it is enough to prove (\ref{e:d+1}) when $B_1,B_2$ are orthogonal unit segments, $B_i=[0,b_i]$, $i=1,2$. In this case (\ref{e:d+1}) becomes
\begin{equation*}
 \vol(A)\,\vol\bigl( P_{[b_1,b_2]^\perp} A \bigr)
\;\leq\; 
\vol\bigl( P_{b_1^\perp} A \bigr)\,\vol\bigl( P_{b_2^\perp} A \bigr). 
\end{equation*}

By Corollary~\ref{C:dual-inequalities}, the above inequality holds if and only if it holds for any $A\in\Z_m^3$ and orthonormal $\{b_1,b_2\}$ in $\R^3$, which was already established in~\cite{FMMZ}.
\end{proof}

\section{Reduction to the Hypercube Inequality}\label{sec:HIvsPI}

%\Katja{we need to change the name of the sec.}

%In this section we show that the validity of the Hypercube inequality \eqref{eq:hypercube_ineq} would imply the validity of the Projection Inequality \eqref{eq:projection}. 

In this section we focus on our conjectured inequality \eqref{eq:partial_polymatroid}, that is,  the conjecture that $f(I) = \log \vol P_I(A)$ is a partial polymatroid on the the family of index sets 
    \[
    \{\{n\},\{1,n\},\{2,n\},\ldots, \{n-1,n\},\{1,2,\ldots,n\}\}.
    \]
%about , which is a particular case of the dual local Loomis-Whitnes-type inequality  (\ref{eq:projection_dual-2}) when $k=n-1$ and $c_{k,n,m}=1$. 
%Without loss of generality, we assume that $\{b_1,\dots, b_n\}$ is the standard basis for $\R^n$ and, hence, the projections involved are coordinate projections. 
We show that inequality \eqref{eq:partial_polymatroid} can be reduced to the Hypercube Inequality \eqref{eq:hypercube_ineq}. 

First, \eqref{eq:partial_polymatroid} is equivalent to
\begin{equation}\label{eq:projection_dual-3}
\vol(A) \, \vol\bigl( P_{[e_1,\dots, e_{n-1}]^\perp} A \bigr)^{n-2}
\;\le\; 
\prod_{\ell=1}^{n-1} \vol\bigl( P_{[e_1 \dots, \hat e_\ell,\dots, e_{n-1}]^\perp} A \bigr).
\end{equation}
Let $A\in\R^{n\times m}$ be the matrix corresponding to the zonotope $A$ and let 
$A_{S,\ast}$ denote the submatrix of $A$ consisting of rows of $A$ indexed by
$S\subset[n]$. Then $A_{\{n\},\ast}$ and $A_{\{\ell,n\},\ast}$ are the matrices corresponding to the zonotopes $P_{[e_1,\dots, e_{n-1}]^\perp} A$ and $P_{[e_1 \dots, \hat e_\ell,\dots, e_{n-1}]^\perp}A$, respectively. Thus, (\ref{eq:projection_dual-3}) becomes
\begin{equation}\label{eq:proj_matrix} 
	\vol(A) \vol(A_{\{n\},\ast})^{n-2} \le \prod_{\ell=1}^{n-1}\vol(A_{\{\ell,n\},\ast}).  
\end{equation}

%Using the coordinate projection $(y_i)_{i \in [d]} \mapsto (y_i)_{i \in S}$ with $S \subseteq [n]$, to map the segments $[0,a_1],\ldots,[0,a_m]$ to dimension $|S|$, gives us a lower dimensional zonotope $Z(A_{S,\ast})$. Thus, the Projection inequality \eqref{eq:projection} states that for all $A \in \R^{m \times n}$ holds
%\begin{equation}\label{eq:proj_matrix} 
%	\vol(Z(A)) \cdot \vol(Z(A_{1,\ast}))^{n-2} \le \prod_{k=2}^n \vol(Z(A_{\{1,k\},\ast} )).  
%\end{equation} 
%
%If this conjecture is true, then it must be valid for all matrices $A \in \mathbb{R}^{m \times n}$ and for arbitrary dimensions $m,n \in \mathbb{N}$. Our approach is to reduce it to certain special cases. The inequality is expressed in terms of the absolute values of minors of $A$ and therefore has the structural properties of the determinant, which will play a key role in the analysis.

\subsection{First reduction of matrix $A \in \mathbb{R}^{m \times n}$}
In this step we show that the last row of $A$ can be replaced, without loss of generality, by $\onevec_m$. 

We first observe that we may assume that the last row of $A$ has non-negative entries, since replacing any column $a_i$ by $-a_i$ does not affect the volumes of the zonotopes involved and allows us to eliminate negative signs in the last row.

After this, we may further assume that the last row consists of strictly positive entries. Indeed, the remaining cases can be recovered by the limit argument, using the continuity of the quantities appearing in \eqref{eq:proj_matrix}.  

Under this assumption we can express the $i$-th column $a_i$ as $a_i = x_i \tilde a_i$ with $x_i \in \R_+$ being the last component of $a_i$ and $\tilde a_i = a_i / x_i \in \R^{n-1}\times \{1 \} $. 

After this reformulation, setting $\tilde A = (\tilde a_1,\ldots,\tilde a_m)$, inequality \eqref{eq:proj_matrix} becomes
\[
f(x,\tilde A) \, f(x,\tilde A_{\{n\},\ast})^{n-2} \le \prod_{\ell=1}^{n-1} f(x,\tilde A_{\{\ell,n\},\ast}),
\]
where for a matrix $M$ we use $f(x,M)$ to denote the volume polynomial of the segments defined by the columns of $M$.
%and let $p_S(x,\tilde A) := p(x,\tilde A_{S,\ast})$.
Since $f(x,\tilde A_{\{n\},\ast}) = \sum_{i=1}^m x_i$, we rewrite the inequality as
\begin{equation}\label{eq:proj_matrix_B}
f(x,\tilde A) \left(\sum_{i=1}^m x_i\right)^{n-2} \!\le\, \prod_{\ell=1}^{n-1} f(x,\tilde A_{\{\ell,n\},\ast}). 
\end{equation}

The advantage of this step is that it separates the dependence on the variables in \eqref{eq:proj_matrix}. 
It is now clear that the expression in \eqref{eq:proj_matrix_B} is a polynomial in the $m$ variables $x_i \in \mathbb{R}_+^m$, while the remaining dependence is on the $m(n-1)$ entries of the matrix $\tilde A$. % (corresponding to rows $2,\ldots,n$), which we need to analyze.

\subsection{Second reduction of matrix $A \in \mathbb{R}^{m \times n}$}
In this step we further reduce the matrix $A$ to the case where the last row is $\onevec_m$, and the remaining rows have entries in $\{0,1\}$.

A finite collection $\mathcal{P} = \{P_1,\dots,P_M\}$ of $m$-dimensional polyhedra whose union equals $\mathbb{R}^m$ is called a \emph{polyhedral subdivision} of $\mathbb{R}^m$ provided that for any two distinct members $P_i, P_j \in \mathcal{P}$ the intersection $P_i \cap P_j$ is either empty or a common proper face of both $P_i$ and $P_j$. A set $G \subseteq \mathbb{R}^m$ is said to \emph{generate} $\mathcal{P}$ if every $P_i$ can be expressed as the convex hull of some subset $G_i \subseteq G$.

For instance, if each $P_i$ is pointed, then the union of all vertices and all unbounded edges of the polyhedra in $\mathcal{P}$ forms a generating set for $\mathcal{P}$.

A function $g : \mathbb{R}^m \to \mathbb{R}$ is called \emph{piecewise affine} if there exists a polyhedral subdivision $\mathcal{P} = \{P_1,\dots,P_M\}$ of $\mathbb{R}^m$ such that $g$ is affine on each $P_i$, $i=1,\dots,M$. In that situation we also say that $\mathcal{P}$ is \emph{compatible} with $g$.

We call a function $f : \mathbb{R}^{n_1} \times \cdots \times \mathbb{R}^{n_k} \to \mathbb{R}$ \emph{multi-convex} if it is convex in each argument when the remaining arguments are held fixed.

We recall a key lemma from \cite{AS}, which was shown for dimension $n=3$, and can be extended straightforwardly to higher dimensions as follows.

\begin{lemma}\label{prop:convex_lemma}
Let $P_1,\ldots,P_k$ be polyhedral subdivisions of $\mathbb{R}^{n_1},\ldots,\mathbb{R}^{n_k}$ with generating sets $G_1,\ldots,G_k$, respectively. Let
\[
f:\mathbb{R}^{n_1}\times\cdots\times\mathbb{R}^{n_k} \to \mathbb{R}
\]
be a multi-convex function, and
\[
g_i:\mathbb{R}^{n_i}\to\mathbb{R}, \quad i=1,\ldots,k,
\]
be piecewise affine functions such that each subdivision $P_i$ is consistent with $g_i$.
Then the inequality
\[
f(y_1,\ldots,y_k) \le g_1(y_1)\cdots g_k(y_k)
\]
holds for all $(y_1,\ldots,y_k)\in \mathbb{R}^{n_1}\times\cdots\times\mathbb{R}^{n_k}$
if and only if it holds for all
\[
(y_1,\ldots,y_k)\in G_1\times\cdots\times G_k.
\]
\end{lemma}

\begin{proof}
Since the ``only if'' direction is trivial, we assume
\[
f(y_1,\ldots,y_k) \le g_1(y_1)\cdots g_k(y_k)
\]
for every $(y_1,\ldots,y_k)\in G_1\times\cdots\times G_k$. Take arbitrary points
$y_i\in\mathbb{R}^{n_i}$, $i=1,\ldots,k$. For each $i$ choose a cell $P_i\in\mathcal{P}_i$ with $y_i\in P_i$ and write $P_i=\operatorname{conv}(G_i')$ for some finite subset $G_i'\subseteq G_i$. Then $y_i$ can be expressed as a convex combination
\[
y_i = \sum_{t_i\in T_i} \lambda_{t_i}^{i} \, y_{t_i}^{i},
\qquad
\sum_{t_i\in T_i}\lambda_{t_i}^{i} = 1,\quad \lambda_{t_i}^{i}>0,\quad y_{t_i}^{i}\in G_i', \quad  i=1,\ldots,k, 
\]
where $T_i$ are finite index sets.

By multi‑convexity of $f$ (convexity in each argument separately) we obtain
\begin{align*}
f(y_1,\ldots,y_k)
&= f\!\left(\sum_{t_1}\lambda_{t_1}^{1}y_{t_1}^{1},\;
          \sum_{t_2}\lambda_{t_2}^{2}y_{t_2}^{2},\;\ldots,\;
          \sum_{t_k}\lambda_{t_k}^{k}y_{t_k}^{k}\right) \\
&\le \sum_{t_1,\ldots,t_k} \Bigl(\prod_{i=1}^k \lambda_{t_i}^{i}\Bigr)\,
      f\!\left(y_{t_1}^{1},y_{t_2}^{2},\ldots,y_{t_k}^{k}\right).
\end{align*}
Applying the hypothesis to each tuple $(y_{t_1}^{1},\ldots,y_{t_k}^{k})\in G_1\times\cdots\times G_k$ gives
\[
f\!\left(y_{t_1}^{1},\ldots,y_{t_k}^{k}\right)
\le \prod_{i=1}^k g_i\!\left(y_{t_i}^{i}\right).
\]
Therefore,
\begin{align*}
f(y_1,\ldots,y_k)
&\le \sum_{t_1,\ldots,t_k} \Bigl(\prod_{i=1}^k \lambda_{t_i}^{i}\Bigr)\;
      \prod_{i=1}^k g_i\!\left(y_{t_i}^{i}\right) \\
&= \prod_{i=1}^k \left( \sum_{t_i\in T_i} \lambda_{t_i}^{i}\, g_i\!\left(y_{t_i}^{i}\right) \right).
\end{align*}
Since each $g_i$ is affine on the cell $P_i$ (the subdivision $\mathcal{P}_i$ is consistent with $g_i$) and 
\[
\sum_{t_i}\lambda_{t_i}^{i}y_{t_i}^{i}=y_i,
\]
we have
\[
\sum_{t_i\in T_i} \lambda_{t_i}^{i}\, g_i\!\left(y_{t_i}^{i}\right)
= g_i\!\left(\sum_{t_i\in T_i} \lambda_{t_i}^{i}y_{t_i}^{i}\right)
= g_i(y_i).
\]
Thus,
\[
f(y_1,\ldots,y_k) \le \prod_{i=1}^k g_i(y_i),
\]
which completes the proof. 
\end{proof}

We now show how  Lemma~\eqref{prop:convex_lemma} can be applied to 
\eqref{eq:proj_matrix_B}, yielding a further reduction. In what follows we abbreviate the absolute value of the determinant by $\Det$.

We view the functions in \eqref{eq:proj_matrix_B} %of the variables $x_1,\ldots,x_m$ 
as functions of the first $n-1$ rows of $\tilde A$, which we denote by $y_1,\dots, y_{n-1}$. 
%Introducing the vectors $y_k := \tilde A_{k+1,\ast} \in \mathbb{R}^m$, $k=1,\dots, n-1$, we have 
Then the left-hand side of \eqref{eq:proj_matrix_B} is
\begin{equation}\label{e:n-by-n}
f(y_1,\ldots,y_{n-1}) := f(x,\tilde A)
= \sum_{I \in \binom{[m]}{n}} \Det\big((y_1)_I,\ldots,(y_{n-1})_I,\onevec_n\big)\prod_{i \in I} x_i.
\end{equation}

Similarly, for the right-hand side of \eqref{eq:proj_matrix_B} we have
\begin{equation}\label{e:2by2}
g(y_{\ell}):=f(x,\tilde A_{\{\ell,n\},\ast})=
\sum_{\{i,j\} \in \binom{[m]}{2}}
\Det \begin{bmatrix}
(y_{\ell})_i & (y_{\ell})_j \\
1 & 1
\end{bmatrix}x_i x_j
\end{equation}
%
%Observe that the right-hand side of \eqref{eq:proj_matrix_B} is a product of factors $p_{1k}(B)$, $k \in \{2,\ldots,n\}$. Since the first row of $B$ equals $\onevec_m$ and its $k$-th row is $y_{k-1}$, the quantity $p_{1k}(B)$ is the volume polynomial of the two-dimensional zonotope generated by the columns of the matrix
%\[
%\begin{bmatrix}
%\onevec_m \\
%y_{k-1}
%\end{bmatrix}.
%\]
%Hence,
%\[
%p_{1k}(x,B)
%= \sum_{\{i,j\} \in \binom{[m]}{2}} x_i x_j
%\Det \begin{bmatrix}
%1 & 1 \\
%(y_{k-1})_i & (y_{k-1})_j
%\end{bmatrix}
%=: g(y_{k-1})
%\]
with $g : \R^m \to \R$ being a continuous piecewise linear function. 
Thus, \eqref{eq:proj_matrix_B} becomes
\begin{equation}\label{xy}
f(y_1,\ldots,y_{n-1}) \left( \sum_{i=1}^m x_i \right)^{n-2} \le\ g(y_1)\cdots g(y_{n-1}).
\end{equation}
Observe that the left-hand side of \eqref{xy} is convex in each variable $y_\ell$, while the right-hand side is a product of functions, each depending on a single $y_\ell$ and being continuous and piecewise linear. Thus, we may apply Lemma~\eqref{prop:convex_lemma}. 

Next we determine a polyhedral subdivision of $\mathbb{R}^m$ that is consistent with $g$. Each function $g(y_{\ell})$ is the sum of functions 
\[
g_{ij}(y_{\ell}) := \bigl| (y_{\ell})_i - (y_{\ell})_j \bigr| \quad \text{with} \quad 1 \le i < j \le m. 
\]
The function $g_{ij}$ is affine on both sides of the hyperplane $X_{ij} := \{x \in \mathbb{R}^m : x_i = x_j\}$ and the function $g= \sum_{i<j} g_{ij}$ is affine on each region of the hyperplane arrangement $\mathcal{X}=\{X_{ij}:1\le i<j\le m\}$. Note that $\mathcal{X}$ is the \emph{braid arrangement}: each of the $m!$ regions of $\mathcal{X}$ corresponds to a way of sorting the values $x_i$, i.e., each region is given by the system of inequalities
\[
x_{\sigma(1)} \le x_{\sigma(2)} \le \cdots \le x_{\sigma(m)},
\]
for some permutation $\sigma \in S_m$. %See Figure \ref{fig:braid3} for an illustration with $m=3$. 
We denote such a region by $R_\sigma$. The polyhedral subdivision $\mathcal{P} = \{ R_\sigma : \sigma\in S_m\}$ is consistent with $g$ and the region $R_\sigma$ is a polyhedral cone with a one‑dimensional lineality space, given by $x_1 = \cdots = x_m$, and whose two‑dimensional faces are given by the conditions
\[
x_{\sigma(1)} = \cdots = x_{\sigma(l)}
\;\le\;
x_{\sigma(l+1)} = \cdots = x_{\sigma(m)},
\quad l = 1,\dots,m-1.
\]
Since the lineality space of $R_\sigma$ is one‑dimensional, $R_\sigma$ is the convex hull of its two‑dimensional faces. Coordinates of points of two-dimensional faces take only two possible values. Therefore, by Lemma~\eqref{prop:convex_lemma}, it suffices to verify \eqref{xy} for $y_{\ell}\in \{\lambda_\ell,\mu_\ell\}^m$ for some fixed $\lambda_\ell,\mu_\ell \in \mathbb{R}$, $\ell=1, \dots, n-1.$

%\begin{figure}
%\vspace{-5cm}
%\begin{center}
%\includegraphics[width=0.7\textwidth]{braid_arrangement .pdf}
%\end{center}
%\vspace{-5cm}
%\caption{The braid arrangement in dimension $3$. The three planes given by the equations $y_1 = y_2$, $y_1 = y_3$ and $y_2 = y_3$ intersect in the line $y_1 = y_2 = y_3$. The planes decompose the space into six cells, each determined by a permutation $\sigma \in S_3$ and defined by the inequalities $y_{\sigma(1)} \le y_{\sigma(2)} \le y_{\sigma(3)}$.}
%\label{fig:braid3}
%%\vspace{-5cm}
%\end{figure}

%Thus, applying Proposition~\eqref{prop:convex_lemma} shows that, for the verification of \eqref{xy}, it suffices to consider the case $y_i \in \{\lambda_i,\lambda_i'\}^m$, where $\lambda_i,\lambda_i' \in \mathbb{R}$. 

%\Ivan{Stopped here. To be continued...}

Observe that if for some $\ell$, one has $\lambda_\ell = \mu_\ell$, then $y_\ell= \lambda_\ell \onevec_m$, and \eqref{xy} reduces to the trivial inequality $0 \le 0$. Otherwise, if $\lambda_\ell \neq \mu_\ell$ for all $\ell\in [n-1]$, we may further reduce to the case 
$\{\lambda_\ell,  \mu_\ell\}= \{0,1\}$ by the following simple observation.

\begin{lemma} \label{box:to:cube} 
Consider the map $\phi : \{0,1\}^{n-1} \to \{\lambda_\ell,  \mu_\ell\}^{n-1}$ defined by
\[
\phi(z_1,\ldots,z_{n-1}) = \bigl((1-z_\ell)\lambda_\ell + z_\ell \mu_\ell\ :\ \ell \in [n-1]\bigr).
\]
Then for any vectors $v_1,\dots, v_{n-1}\in \{0,1\}^{n-1}$ we have
\[
\Det\bigl(\phi(v_1),\ldots,\phi(v_{n-1}), \onevec_n\bigr)
= 
\Det\bigl(v_1,\ldots,v_{n-1}, \onevec_n\bigr)\prod_{\ell=1}^{n-1} |\lambda_\ell - \mu_\ell|.
\]
\end{lemma} 
\begin{proof} This follows from the multilinearity of the determinant.
%	Use the fact that for $z = (z_1,\ldots,z_n)$ 
%	\[
%		\underbrace{ 
%		\left[
%		\begin{array}{c|ccc}
%			1 & 0 & \cdots & 0 \\
%			\hline 
%			\lambda_1 & \lambda_1' - \lambda_1 & & \\
%			\vdots & & \ddots &  \\
%			\lambda_n &  & & \lambda_n'-\lambda_n
%		\end{array}		\right] }_{=:T}
%		\begin{bmatrix} 1 \\ z \end{bmatrix} = \begin{bmatrix} 1 \\ \phi(z) \end{bmatrix},
%	\]
%	where $\det(T) = \prod_{i=1}^n (\lambda_i' -\lambda_i).$
\end{proof}

Applying Lemma~\ref{box:to:cube} to determinants of size $n$ (on the left-hand side) and of size $2$ (on the right-hand side), we reduce \eqref{xy} to the case where $y_\ell \in \{0,1\}^m$. 
%
%At this point, we want to go back from the interpretation in terms of the vectors $y_1,\ldots,y_{n-1}$ to the interpretation in terms of the matrix $\tilde A = (\tilde a_1,\ldots,\tilde a_m)$ and its columns $\tilde a_1,\ldots,\tilde a_m$. After the reduction, the columns $\tilde a_i$ belong to $\{1\} \times \{0,1\}^{n-1}$. In other words, we go back from the row perspective to the column perspective. That means, 
Equivalently, 
we have reduced \eqref{eq:proj_matrix_B} to the case $\tilde A = (\tilde a_1,\ldots,\tilde a_m)$ with $\tilde a_i \in \{0,1\}^{n-1}\times \{1\}$. 

\subsection{Third reduction of matrix $A \in \mathbb{R}^{m \times n}$}
In this step we reduce the number of columns of $A \in \mathbb{R}^{m \times n}$ from an arbitrary $m \geq n+1$ to $2^{n-1}$.

If the columns of $A$ lie in $\{0,1\}^{n-1}\times \{1\}$ then $A$ has at most $2^{n-1}$ distinct columns. We use the following elementary property of the volume polynomial: If
\[
f(x_1,\ldots,x_m) = \vol(x_1 K_1 + \cdots + x_m K_m),
\]
and some of the bodies coincide, say $K_1 = \cdots = K_r$ for some $1 <r \leq m$, then
\[
f(x_1,\ldots,x_m) = f(x_1 + \cdots + x_r, 0,\ldots,0, x_{r+1},\ldots,x_m).
\]

Thus, we may group the variables $x_i$ according to the groups of repeated columns of $A$,
%the vectors $\tilde a_i \in \{1\} \times \{0,1\}^{n-1}$ and sum the variables within each group,  
thereby reducing the number of variables to $2^{n-1}$. This implies that
it suffices to verify inequality \eqref{eq:proj_matrix_B} for all $\{0,1\}$-matrices
of dimension $n\times 2^{n-1}$.
%More precisely, for
%\[
%u \in U := \{1\} \times \{0,1\}^{n-1},
%\]
%we set
%\[
%s(u) := \sum_{i \in I(u)} x_i, 
%\quad \text{where} \quad
%I(u) := \{ i \in [m] : \tilde a_i = u \}.
%\]
%This allows us to rewrite $p(x,B)$ as $p(s,U)$, where $s = (s(u))_{u \in U} \in \mathbb{R}_+^{2^{n-1}}$.
%
%Here, $U$ is interpreted as an $n \times 2^{n-1}$ matrix whose first row is $\onevec$, and whose remaining $(n-1)$ rows consist of all $0/1$ vectors of length $n-1$, listed as columns.

%The same reduction can also be applied to $p_1(x,B)$ and $p_{1k}(x,B)$. This reduces \eqref{eq:proj_matrix_B} from the case of an arbitrary $m$ to the case $m = 2^{n-1}$, and where $A$ is fixed to be $U$, an $n \times 2^{n-1}$ matrix.
%
%In other words, after applying the described reduction, in order to prove inequality \eqref{eq:proj_matrix_B} for all choices of $A \in \mathbb{R}^{m \times n}$ with arbitrary $m$, it suffices to restrict to a single special case. Thus, it remains to verify that the function we call the \emph{hypercube polynomial}
%\begin{equation}\label{eq:hypercube_pol}
%f(s) := \prod_{k=2}^n p_{1,k}(s,U) - p(s,U)\, p_1(s,U)^{n-2} \quad \text{with} \quad U = \{1\} \times \{0,1\}^{n-1}
%\end{equation}
%is non-negative for all $s \in \mathbb{R}_+^{2^{n-1}}$. 

\subsection{Hypercube Inequality} We now restate \eqref{eq:proj_matrix_B} as a polynomial inequality
in $m=2^{n-1}$ variables. Its formulation uses geometry and combinatorics of the $(n-1)$-dimensional unit cube, so we call it a {\it Hypercube Inequality}. To simplify notation we set $d=n-1$ and denote the $d$-dimensional unit cube by $[0,1]^d$.

To each vertex $v\in\{0,1\}^d$ of $[0,1]^d$ we assign a variable $x_v$. For any collection $S$ of $d+1$ vertices
of $[0,1]^d$ we let $\Vol(S)$ denote the normalized volume of the convex hull of $S$, i.e., $\Vol(S)=d!\vol(\conv(S))$. Note that if $S=\{v_0,\dots, v_d\}$ then 
$$\Vol(S)=\Det
\begin{pmatrix}v_0 & v_1 & \dots & v_d\\
1 & 1 & \dots & 1
\end{pmatrix}
$$ 
and $\Vol(S)$ is a positive integer whenever $S$ is affinely independent.

Now let $\tilde A$ be the $(d+1)\times 2^d$ matrix with columns $(v, 1)$ over all possible $v\in \{0,1\}^d$. By \eqref{e:n-by-n} we have
$$f(x,\tilde A) =\sum_{S \in \binom{\{0,1\}^d}{d+1}} 
            \Vol(S)\prod_{v \in S} x_v.$$
Also, from \eqref{e:2by2} we obtain
$$
f(x,\tilde A_{\{\ell,n\},\ast})=\sum_{v_\ell\neq v'_\ell}x_{v}x_{v'}=\Big(\sum_{v_\ell=0}x_v\Big)\Big(\sum_{v_\ell=1}x_v\Big).
$$
Note that in the right-hand side we sup up the variables over the vertices of the facets determined by $v_\ell=0$ and $v_\ell=1$, respectively.
Therefore, \eqref{eq:proj_matrix_B} becomes
\begin{equation}\label{eq:hypercube_ineq-2}
	\left( 
        \sum_{S \in \binom{\{0,1\}^d}{d+1}} 
            \Vol(S)\prod_{v \in S} x_v
            \right)
        \left( 
        \sum_{v \in \{0,1\}^d} x_v 
    \right)^{d-1}  
    \le 
    \prod_{\substack{F\subset [0,1]^d\\ F\, \text{facet}}} 
    \sum_{v \in F \cap \{0,1\}^d} 
        x_v,
\end{equation}

The difference of the right- and left-hand sides of \eqref{eq:proj_matrix_B} is
a homogeneous polynomial of total degree $2d$ and is invariant under the symmetry group of the cube.
%$\Sym([0,1]^d)\cong {\mathbf S}_4\times{\mathbf S}_2$. 
We denote this polynomial by $H_d(x)$ and call it the {\it hypercube polynomial}.

\begin{rmk}\label{rem:grassmannian}
We would like to mention that the reduction from this section has an interpretation in terms of the absolute Grassmannian
$\operatorname{Abs}(\operatorname{Gr}(d,n))$, where $\operatorname{Abs}$ denotes the coordinate-wise absolute value map.
We write $q_I=|p_I|$ for the absolute value of the Plücker coordinate.
Then, inequality \eqref{eq:proj_matrix} can be formulated as
\[
\left( \sum_{I \in \binom{[m]}{n}} q_I \right) \left( \sum_{i=1}^m q_{\{i, m+1,\dots, m+n-1\}} \right)^{n-2} \leq \prod_{1 \leq i_1 < \ldots < i_{n-2} \leq n-1}
\left( \sum_{1 \leq i < j \leq m} q_{\{i, j, m+i_1,\dots, m+i_{n-2}\}} \right)
\]
for all $q \in \operatorname{Abs}(\operatorname{Gr}(n,m+n-1))$. 
Moreover, the reduction shows, in particular, that it suffices to verify this inequality for the fixed value $m = 2^{n-1}$, rather than for arbitrary $m \ge n$.
\end{rmk}

\section{Explicit form of Hypercube Inequality \eqref{eq:hypercube_ineq} for $d=2$ and $d=3$}
In this section we give an explicit description of the Hypercube Inequality and the hypercube
polynomial $H_d(x)$ in the smallest cases $d=2,3$. To simplify notation
we use decimal indices for our variables $x_i$, for $i=0,\dots, 2^d$ instead of the binary indices $x_v$, for $v\in\{0,1\}^d$ that were introduced in the previous section.

It was observed in \cite{AS} (and is not hard to verify directly) that $H_2(x)$ is a perfect square.
Indeed, every triangle with vertices in $\{0,1\}^2$ has normalized $2$-dimensional volume 1, hence, we
have
\begin{equation*}\label{eq:HI_d=2}
H_2(x) = (x_0+x_1)(x_0+x_2)(x_2+x_3)(x_1+x_3)-  \left(\sum_{1 \le 0<j<k \le 3} x_i x_j x_k\right)
\Big(\sum_{i=0}^{3} x_i\Big)
= (x_0 x_3 - x_1 x_2)^2. 
\end{equation*}
This shows that the Hypercube Inequality \eqref{eq:hypercube_ineq} holds for $d=2$. It will be convenient to express it as follows

\begin{equation}\label{e:HI_d=2}
\Big(\sum_{i=0}^{3} x_i\Big)\big((x_2+x_3)x_0 x_1 + (x_0+x_1)x_2x_3\big)
\leq
(x_0+x_1)(x_0+x_2)(x_2+x_3)(x_1+x_3). 
\end{equation}

We now describe the Hypercube Inequality \eqref{eq:hypercube_ineq} in the case $d=3$. In the left-hand side, we have non-degenerate simplices with vertices among $\{x_0,\dots,x_7\}$. Any such simplex must contain vertices from both the lower facet $\operatorname{conv}(\{x_0,x_1,x_2,x_3\})$ and the upper facet $\operatorname{conv}(\{x_4,x_5,x_6,x_7\})$. Hence, the following cases may occur (see Figure~\ref{fig:simplex}):

\begin{enumerate}[(i): ]
    \item \textbf{One vertex from the lower facet and three from the upper facet, or vice versa.}  
    The contribution of these simplices to the left‑hand side of the inequality is
    \[
    \Bigl(\sum_{i=0}^{3} x_i\Bigr)\Bigl(\sum_{4\le i<j<k\le 7} x_i x_j x_k\Bigr)
      \;+\;
    \Bigl(\sum_{i=4}^{7} x_i\Bigr)\Bigl(\sum_{1\le 0<j<k\le 3} x_i x_j x_k\Bigr).
    \]
    An example of such a simplex is shown in Type (i) of Figure~\ref{fig:simplex}.

    \item \textbf{Two vertices from the lower facet and two from the upper facet, not all lying in a common plane.}  
    One may take an edge from the upper facet together with a non‑parallel edge from the lower facet, or an edge from the upper facet together with a diagonal from the lower facet (the symmetric cases are included automatically). These geometrically distinct possibilities are expressed by the products
\begin{align*}
&(x_0+x_1)(x_2+x_3)(x_4x_5+x_6x_7),\\
&(x_0+x_2)(x_1+x_3)(x_4x_6+x_5x_7),\\
&(x_0+x_3)(x_1+x_2)(x_5x_6+x_4x_7).
\end{align*}
    Examples are shown in the second and third subfigures of Figure~\ref{fig:simplex}.
    
    The remaining possibility, a diagonal from the lower facet together with a non‑parallel diagonal from the upper facet, yields the additional terms
    \[
    2x_0x_3x_5x_6 + 2x_1x_2x_4x_7,
    \]
    and is illustrated in the fourth subfigure of Figure~\ref{fig:simplex}.
\end{enumerate}

%We now describe the hypercube inequality \eqref{eq:hypercube_ineq} in the case $d=3$. On the left-hand side, we consider non-degenerate simplices with vertices among $\{x_0,\dots,x_7\}$, the vertices of the cube. Any such simplex must contain vertices from both the lower facet $\operatorname{conv}(\{x_0,x_1,x_2,x_3\})$ and the upper facet $\operatorname{conv}(\{x_4,x_5,x_6,x_7\})$. Hence, the following cases may occur (see Figure~\ref{fig:simplex}):
%\begin{enumerate}[(i)]
%    \item one vertex from the lower facet and three from the upper facet, or vice versa;
%    \item two vertices from the lower facet and two from the upper facet, not all lying in a common plane.
%\end{enumerate}

All non-degenerate simplices in the $3$-dimensional cube with vertices among those of the cube have normalized volume $1$, except for the simplices $\operatorname{conv}(\{x_0, x_3, x_5, x_6\})$ and $\operatorname{conv}(\{x_1, x_2, x_4, x_7\})$, which each have normalized volume $2$.

\begin{figure}
\begin{subfigure}{0.45\textwidth}
\centering
\begin{tikzpicture}[scale=2, x={(1.0298,0.03356)}, y={(0,1)}, z={(-0.2963,-0.3835)}]

\coordinate (s2) at (0,0,0);
\coordinate (s4) at (1,0,0);
\coordinate (s8) at (1,1,0);
\coordinate (s6) at (0,1,0);

\coordinate (s1) at (0,0,1);
\coordinate (s3) at (1,0,1);
\coordinate (s7) at (1,1,1);
\coordinate (s5) at (0,1,1);

\draw[dashed] (s1) -- (s2)-- (s4);
\draw (s4) -- (s3) -- (s1);

\draw[dred, thick] (s6) -- (s5) -- (s7) -- (s6); 
\draw[dred, thick] (s5) -- (s7) -- (s4) -- (s5); 
\draw[dred, thick] (s5) -- (s6) -- (s4) -- (s5);

\fill[dred,opacity=0.15]
(s6) -- (s5) -- (s7) -- (s6);
\fill[dred,opacity=0.15]
(s5) -- (s7) -- (s4) -- (s5); 
\fill[dred,opacity=0.15]
(s5) -- (s6) -- (s4) -- (s5);

\draw (s5) -- (s6) -- (s8) -- (s7) -- cycle;

\draw (s1) -- (s5);
\draw[dashed] (s2) -- (s6);
\draw (s3) -- (s7);
\draw (s4) -- (s8);

\fill[black] (s4) circle (0.03);
\fill[black] (s5) circle (0.03);
\fill[black] (s6) circle (0.03);
\fill[black] (s7) circle (0.03);

\node at (s1) [below left] {$x_0$};
\node at (s2) [below right] {$x_1$};
\node at (1.1,-0.15,1) {$x_2$};
\node at (1.1,-0.15,0) {$x_3$};

\node at (s5) [below left] {$x_4$};
\node at (s6) [below right] {$x_5$};
\node at (1.15,0.9,1) {$x_6$};
\node at (1.15,0.9,0) {$x_7$};
\end{tikzpicture}
\caption{Type (i): $\Vol(\{x_3,x_4,x_5, x_6\})=1$}
\vspace{10mm}
\end{subfigure}
\hfill
\begin{subfigure}{0.45\textwidth}
\centering
\begin{tikzpicture}[scale=2, x={(1.0298,0.03356)}, y={(0,1)}, z={(-0.2963,-0.3835)}]

\coordinate (s2) at (0,0,0);
\coordinate (s4) at (1,0,0);
\coordinate (s8) at (1,1,0);
\coordinate (s6) at (0,1,0);

\coordinate (s1) at (0,0,1);
\coordinate (s3) at (1,0,1);
\coordinate (s7) at (1,1,1);
\coordinate (s5) at (0,1,1);

\draw[dashed] (s1) -- (s2)-- (s4);
\draw (s4) -- (s3) -- (s1);

\draw (s5) -- (s6) -- (s8) -- (s7) -- cycle;

\draw (s1) -- (s5);
\draw[dashed] (s2) -- (s6);
\draw (s3) -- (s7);
\draw (s4) -- (s8);

\draw[gold, thick] (s3) -- (s4) -- (s6) -- (s3);
\draw[gold, thick] (s3) -- (s4) -- (s8) -- (s3);
\draw[gold, thick] (s3) -- (s6) -- (s8) -- (s3);

\fill[orange,opacity=0.15] (s3) -- (s4) -- (s6) -- (s3);
\fill[orange,opacity=0.15] (s3) -- (s4) -- (s8) -- (s3);
\fill[orange,opacity=0.15] (s3) -- (s6) -- (s8) -- (s3);

\fill[black] (s3) circle (0.03);
\fill[black] (s4) circle (0.03);
\fill[black] (s6) circle (0.03);
\fill[black] (s8) circle (0.03);

\node at (s1) [below left] {$x_0$};
\node at (s2) [below right] {$x_1$};
\node at (1.1,-0.15,1) {$x_2$};
\node at (1.1,-0.15,0) {$x_3$};

\node at (s5) [below left] {$x_4$};
\node at (s6) [below right] {$x_5$};
\node at (1.15,0.9,1) {$x_6$};
\node at (1.15,0.9,0) {$x_7$};
\end{tikzpicture}
\caption{Type (ii): $\Vol(\{x_2,x_3,x_5, x_7\})=1$}
\vspace{10mm}
\end{subfigure}
\hfill
\begin{subfigure}{0.45\textwidth}
\centering
\begin{tikzpicture}[scale=2, x={(1.0298,0.03356)}, y={(0,1)}, z={(-0.2963,-0.3835)}]

\coordinate (s2) at (0,0,0);
\coordinate (s4) at (1,0,0);
\coordinate (s8) at (1,1,0);
\coordinate (s6) at (0,1,0);

\coordinate (s1) at (0,0,1);
\coordinate (s3) at (1,0,1);
\coordinate (s7) at (1,1,1);
\coordinate (s5) at (0,1,1);

\draw[dashed] (s1) -- (s2)-- (s4);
\draw (s4) -- (s3) -- (s1);

\draw (s5) -- (s6) -- (s8) -- (s7) -- cycle;

\draw (s1) -- (s5);
\draw[dashed] (s2) -- (s6);
\draw (s3) -- (s7);
\draw (s4) -- (s8);

\draw[gold, thick] (s2) -- (s3) -- (s6) -- (s2);
\draw[gold, thick] (s2) -- (s3) -- (s8) -- (s2);
\draw[gold, thick] (s2) -- (s6) -- (s8) -- (s2);

\fill[orange,opacity=0.15] (s2) -- (s3) -- (s6) -- (s2);
\fill[orange,opacity=0.15] (s2) -- (s3) -- (s8) -- (s2);
\fill[orange,opacity=0.15] (s2) -- (s6) -- (s8) -- (s2);

\fill[black] (s2) circle (0.03);
\fill[black] (s3) circle (0.03);
\fill[black] (s6) circle (0.03);
\fill[black] (s8) circle (0.03);

\node at (s1) [below left] {$x_0$};
\node at (s2) [below right] {$x_1$};
\node at (1.1,-0.15,1) {$x_2$};
\node at (1.1,-0.15,0) {$x_3$};

\node at (s5) [below left] {$x_4$};
\node at (s6) [below right] {$x_5$};
\node at (1.15,0.9,1) {$x_6$};
\node at (1.15,0.9,0) {$x_7$};
\end{tikzpicture}
\caption{Type (ii): $\Vol(\{x_1,x_2,x_5, x_7\})=1$}
\end{subfigure}
\hfill
\begin{subfigure}{0.45\textwidth}
\centering
\begin{tikzpicture}[scale=2, x={(1.0298,0.03356)}, y={(0,1)}, z={(-0.2963,-0.3835)}]

\coordinate (s2) at (0,0,0);
\coordinate (s4) at (1,0,0);
\coordinate (s8) at (1,1,0);
\coordinate (s6) at (0,1,0);

\coordinate (s1) at (0,0,1);
\coordinate (s3) at (1,0,1);
\coordinate (s7) at (1,1,1);
\coordinate (s5) at (0,1,1);

\draw[dashed] (s1) -- (s2)-- (s4);
\draw (s4) -- (s3) -- (s1);

\draw (s5) -- (s6) -- (s8) -- (s7) -- cycle;

\draw (s1) -- (s5);
\draw[dashed] (s2) -- (s6);
\draw (s3) -- (s7);
\draw (s4) -- (s8);

\draw[dred, thick] (s2) -- (s3) -- (s5) -- (s2); 
\draw[dred, thick] (s2) -- (s3) -- (s8) -- (s2);
\draw[dred, thick] (s3) -- (s5) -- (s8) -- (s3);

\fill[dred,opacity=0.15] (s2) -- (s3) -- (s5) -- (s2); 
\fill[dred,opacity=0.15] (s2) -- (s3) -- (s8) -- (s2);
\fill[dred,opacity=0.15] (s3) -- (s5) -- (s8) -- (s3);

\fill[black] (s2) circle (0.03);
\fill[black] (s3) circle (0.03);
\fill[black] (s5) circle (0.03);
\fill[black] (s8) circle (0.03);

\node at (s1) [below left] {$x_0$};
\node at (s2) [below right] {$x_1$};
\node at (1.1,-0.15,1) {$x_2$};
\node at (1.1,-0.15,0) {$x_3$};

\node at (s5) [below left] {$x_4$};
\node at (s6) [below right] {$x_5$};
\node at (1.15,0.9,1) {$x_6$};
\node at (1.15,0.9,0) {$x_7$};
\end{tikzpicture}
\caption{Type (ii): $\Vol(\{x_1,x_2,x_4, x_7\})=2$}
\end{subfigure}
\caption{Simplices in the 3-cube classified by vertex distribution across the lower and upper facets}
\label{fig:simplex}
\end{figure}

%The main result of this work is the verification of the hypercube inequality \eqref{eq:hypercube_ineq} for $d=3$, thereby confirming the validity of \eqref{eq:oursubmod} for $n=4$. 

\begin{thm}[Hypercube inequality, $d=3$]\label{thm:HI3}
Let $x_0,\dots,x_7 \ge 0$. Then
\begin{align}\label{eq:HI_d=3}
&\left(\sum_{i=0}^{7} x_i\right)^2 
\Big[
    \Bigl(\sum_{i=0}^{3} x_i\Bigr)\Bigl(\sum_{4\le i<j<k\le 7} x_i x_j x_k\Bigr)+
    \Bigl(\sum_{i=4}^{7} x_i\Bigr)\Bigl(\sum_{1\le 0<j<k\le 3} x_i x_j x_k\Bigr)
\notag \\
&\hspace{5cm} + (x_0+x_1)(x_2+x_3)(x_4x_5+x_6x_7) \notag \\
&\hspace{5cm}  + (x_0+x_2)(x_1+x_3)(x_4x_6+x_5x_7) \notag \\
&\hspace{5cm}  + (x_0+x_3)(x_1+x_2)(x_5x_6+x_4x_7) \notag \\
&\hspace{5cm}  + 2x_0x_3x_5x_6 + 2x_1x_2x_4x_7
\Big]\notag \\
&\hspace{.0cm} \leq 
\Big(\sum_{i=0}^{3} x_i\Big)
\Big(\sum_{i=4}^{7} x_i\Big)
\Big(\sum_{i \in \{0,2,4,6\}} x_i\Big) 
\Big(\sum_{i \in \{1,3,5,7\}} x_i\Big) 
\Big(\sum_{i \in \{0,1,4,5\}} x_i\Big) 
\Big(\sum_{i \in \{2,3,6,7\}} x_i\Big). 
\end{align}
%Equality holds if and only if at least one factor in the product on the right‑hand side is zero. 
\end{thm}

\section{Proof of the Hypercube inequality \eqref{eq:hypercube_ineq} in dimension $d=3$}\label{sec:proof_HI}

The idea of the proof is to combine several hypercube inequalities in dimension two. 
Each edge of the 3-cube defines an ``edge variable'' which is the sum of the variables
corresponding to its endpoints. By projecting the 3-cube along two coordinate directions we obtain two hypercube inequalities in $d=2$ in the corresponding edge variables.
Additionally, we apply the hypercube inequality in $d=2$ to the top and bottom faces of the $3$-cube. In this sense, the proof is an inductive argument from $d=2$ to $d=3$.

%\begin{thm}\label{T:HI_d=3}
%    The hypercube inequality \eqref{eq:hypercube_ineq} 
%    holds for $d=3$ and $x_0,\dots,x_7\in\R_+$.
%\end{thm}

\begin{proof}[Proof of Theorem \ref{thm:HI3}]
Let
\[
m = \sum_{i=0}^3 x_i, \qquad k = \sum_{i=4}^7 x_j
\]
be the sums of the variables for the bottom and top faces of the cube. We introduce the edge variables
\[
\begin{aligned}
z_0 &= x_0 + x_1, \quad z_1 = x_0 + x_2, \quad z_2 = x_2 + x_3, \quad z_3 = x_1 + x_3,\\
z_4 &= x_4 + x_5, \quad z_5 = x_4 + x_6, \quad z_6 = x_6 + x_7, \quad z_7 = x_5 + x_7.
\end{aligned}
\]

Then we can express the right-hand side $h^+$ and the left-hand side $h^-$ of \eqref{eq:HI_d=3} as follows
\begin{align}\label{eq:HI_d=3_simple}
h^+&=(z_0+z_2)(z_0+z_4)(z_2+z_6)(z_1+z_5)(z_3+z_7)(z_5+z_7),\\
h^-&=(m+k)^2 \Big[
k\!\sum_{0 \le i<j<k \le 3} x_i x_j x_k +
m\!\sum_{4 \le i<j<k \le 7} x_i x_j x_k \notag \\
&\quad + z_0z_2(x_4x_5+x_6x_7) \notag + z_1z_3(x_4x_6+x_5x_7) + (x_0+x_3)(x_1+x_2)(x_5x_6+x_4x_7) \notag \\
&\quad + 2x_0x_3x_5x_6 + 2x_1x_2x_4x_7
\Big].\notag 
\end{align}
The latter simplifies to
\begin{align}\label{e:f_minus}
h^-=(m+k)^2 \Big[
&k\!\sum_{0 \le i<j<k \le 3} x_i x_j x_k +
m\!\sum_{4 \le i<j<k \le 7} x_i x_j x_k \notag \\
&\quad + z_0z_2z_5z_7 + z_1z_3z_4z_6- 2(x_0x_3x_4x_7 + x_1x_2x_5x_6)\Big].
\end{align}

It is easy to see that when $m=0$ or $k=0$, both $h^+$ and $h^-$ vanish and, hence, $h^-\leq h^+$ holds trivially. Thus, we will assume that $m$ and $k$ are positive and show $mkh^-\leq mkh^+$.

First, we apply the hypercube inequality for $d=2$ \eqref{e:HI_d=2} to the top and bottom faces of the cube
$$
m\!\sum_{0 \le i<j<k \le 3} x_i x_j x_k \leq z_0z_1z_2z_3\quad\text{and}\quad 
k\!\sum_{4 \le i<j<k \le 7} x_i x_j x_k \leq z_4z_5z_6z_7
$$
to obtain from \eqref{e:f_minus}
\begin{align*}
mkh^-&\leq  (m+k)^2\Big[k^2z_0z_1z_2z_3+m^2z_4z_5z_6z_7\\ 
&\quad +mk(z_0z_2z_5z_7 + z_1z_3z_4z_6)- 2mk(x_0x_3x_4x_7 + x_1x_2x_5x_6)\Big]. 
\end{align*}
Ignoring the negative term and factoring, we get
\begin{equation}\label{e:mkf_minus}
 mkh^-\leq (m+k)^2(kz_0z_2 + mz_4z_6)(kz_1z_3 + mz_5z_7).
\end{equation}

Next, note that $z_0,z_2,z_4,z_6$ correspond to the edges in the $x$-direction and $z_1,z_3,z_5,z_7$ correspond to the edges in the $y$-direction. 
Applying the hypercube inequality for $d=2$ \eqref{e:HI_d=2} to these 4-tuples and
noting that %$m=z_0+z_2=z_1+z_3$ and $k=z_4+z_6=z_5+z_7$
\begin{align*}
m &= z_0 + z_2 = z_1 + z_3, \\
k &= z_4 + z_6 = z_5 + z_7, 
\end{align*}
we obtain
\begin{align*}
(m+k)(kz_0z_2 + mz_4z_6) &\leq mk(z_0 + z_4)(z_2 + z_6),\\
(m+k)(kz_1z_3 + mz_5z_7) &\leq mk(z_1 + z_5)(z_3 + z_7).    
\end{align*}
Taking the product of these inequalities gives
\begin{equation}\label{e:mkf_plus}
 (m+k)^2(kz_0z_2 + mz_4z_6)(kz_1z_3 + mz_5z_7) \leq mkh^+.   
\end{equation}
Combining \eqref{e:mkf_minus} and \eqref{e:mkf_plus}, we get $mkh^-\leq mkh^+$.
\end{proof}

\begin{rem} \label{R:equality_case}
    In the above estimates we used the hypercube inequality in dimension two, as well as ignored the negative term $2mk(x_0x_3x_4x_7 + x_1x_2x_5x_6)$. We saw in \eqref{eq:HI_d=2} that the hypercube polynomial in dimension two is a perfect square. This observation allows us to write an explicit sum-of-squares based certificate for the hypercube polynomial $H_3(x)=h^+-h^-$ scaled by $mk$:
\begin{align*}\label{e:mkf-certificate}
mkH_3(x) =&(m+k)^2\Big[k^2(x_0x_3-x_1x_2)^2+m^2(x_4x_7-x_5x_6)^2
+ 2mk(x_0x_3x_4x_7 + x_1x_2x_5x_6)\Big]\\
 &+ (m+k)\Big[(kz_0z_2 + mz_4z_6)(z_1z_7-z_3z_5)^2
+(kz_1z_3 + mz_5z_7)(z_0z_6-z_2z_4)^2\Big]\notag\\
&+(z_0z_6-z_2z_4)^2(z_1z_7-z_3z_5)^2\notag. 
\end{align*}
\end{rem}

The next result describes the equality case of the hypercube inequality in $d=3$. The proof uses a simple observation: a $2\times 2$ singular matrix with a zero entry must have a zero row or a column. We apply this observation iteratively, first for the vertex variables in a facet and then for the edge variables in a coordinate projection of the cube.

\begin{thm}
    The hypercube inequality \eqref{eq:HI_d=3} holds with equality if and only if for at least one facet of the cube the corresponding $x_i$ are zero. Moreover, the inequality is tight.
\end{thm}

\begin{proof}
First, suppose $x_i=0$ for all $x_i$ in some facet $F$ of the cube. Clearly, this implies that $h^+=0$. Also, since every 3-simplex in the cube contains a vertex in $F$, the left-hand side $h^-$ vanishes as well.

Conversely, suppose $h^+=h^-$. If $m=0$ or $k=0$ there is nothing to show, so we assume that $m,k$ are positive. As we saw in the proof of Theorem~\ref{thm:HI3} (see also Remark~\ref{R:equality_case}), the equality case enforces the corresponding equality cases for $d=2$, hence, the following quadratic relations must hold:
\begin{equation}\label{e:quad-rel}
 x_0x_3-x_1x_2=0,\quad x_4x_7-x_5x_6=0,
\quad z_0z_6-z_2z_4=0,\quad z_1z_7-z_3z_5=0.   
\end{equation}
Additionally, we must have $x_0x_3x_4x_7 = x_1x_2x_5x_6=0$. The latter implies that $x_i=0$ for some $i\in[8]$, and by symmetry we can assume $x_0=0$. Then the first relation in \eqref{e:quad-rel} implies $x_1x_2=0$. Again, by symmetry, we may assume $x_1=0$. We have $z_0=x_0+x_1=0$, so the third relation in \eqref{e:quad-rel} implies $z_2=0$ or $z_4=0$. Since by assumption $m=z_0+z_2>0$, it follows that $z_0=z_4=0$, i.e. $x_0=x_1=x_4=x_5=0$, and we are done.

Finally, we show that the hypercube inequality is tight, that is, there is no constant $0<c<1$ such that $h^-\leq ch^+$ holds in the positive orthant. Consider 
\[
x_0 = x_1 = x_2 = x_3 = 1,\qquad x_4 = x_5 = x_6 = x_7 = \varepsilon,
\]
for some $\varepsilon>0$. Then the ratio becomes
\[
\frac{h^+}{h^-} = \frac{4\cdot 4\varepsilon\cdot(2+2\varepsilon)^4}{(4+4\varepsilon)^2(16\varepsilon+16\varepsilon^3+28\varepsilon^2)}
\rightarrow 1,\ \ \text{as }\ \varepsilon\to 0
\]
showing that the constant $1$ is optimal.
\end{proof}

\begin{rem} 
	In our proof of the copositivity of $H_d(x)$ in the case $d=3$, we implicitly used invariance properties of this polynomial. Note that the theory of invariant polynomials is a  well-developed part of modern algebra. Since $x$ can be interpreted as a tensor of size $\underbrace{2 \times \cdots \times 2}_d$, the hypercube polynomial $H_d(x)$ is a function of a tensor. Another well-known invariant polynomial, which is a function of a tensor, is the hyperdeterminant, see \cite{GKZ}. Polynomial $H_d(x)$ shares the following properties with the square of the hyperdeterminant: (i) If a $\underbrace{2 \times \cdots \times 2}_{d-1}$ layer of $H_d(x)$ is equal to zero, then $H_d(x)$ is zero; (ii)  
$H_d(x)$ is invariant under the natural action of the group of the $2^d \cdot d!$ symmetries of the $d$-dimensional cube. 

	Although $H_2(x)$ is precisely the square of the $2 \times 2$ determinant, the connection to the hyperdeterminant for $d \ge 3$  is yet to be understood. For $d=3$, we would like to mention that the hyperdeterminant is a homogeneous polynomial in the variables of any given $2 \times 2$ layer, but $H_3(x)$ does not have this property. Nevertheless, it is suggestive that the theory laid out in sources like \cite{GKZ}  might be helpful for verifying copositivity %(TODO INTRODUCE COPOSITIVITY) 
    of $H_d(x)$ for higher values of $d$. 
	
%	It would be interesting to understand the symmetry structure of the hypercube polynomials more thoroughly. Currently, we only have a few observations that do not provide a complete clarification but merely suggest certain directions one may explore. 
%	Note that the hypercube polynomial $H_d(x)$ with $d=2$ is the square of the $2 \times 2$ determinant:
%	\[
%		H_2(x) = \det \begin{pmatrix} x_0 & x_1 \\ x_2 & x_3 \end{pmatrix} ^2. 
%	\]
%	This gives rise to various invariance properties of $H_2(x)$. For $d=2$, we can view $x$ as  a matrix, while multiplication of a row or a column of $x$ by $\lambda \in \R$ results into scaling $H_2(x)$ by factor $\lambda^2$. Furthermore, permutation of rows or columns of $x$ keeps $H_2(x)$ invariant, because the determinant is an alternating function with respect to these operations. When we move to the case $d=3$ and consider $H_3(x)$, we can view $x$ as a $2 \times 2 \times 2$ tensor. Here, one has analogous invariance properties. Multiplication of a $2 \times 2$ layer of the tensor $x$ by $\lambda \in \R$ results into scaling $H_3(x)$ by factor $\lambda^2$, while permutation of any two parallel $2 \times 2$ layers  keeps $H_3(x)$ unchanged. So, $H_3(x)$ have similar invariance properties to the square of the $2 \times 2 \times 2$ hyper-determinant. The theory of hyper-determinants is laid out in the book of Gelfand \& Kapranov \& Zelevinskiy. We have been able to find a  natural relation of $H_3(x)$ to the square of hyper-determinant, but investigating such a relation would be an interesting research objective. 
\end{rem} 

Recall that the polynomial $p$ is \emph{copositive} if $p(x) \geq 0$ for all $x \in \mathbb{R}^n_+$.

\begin{rem} 
	Systematic approaches to certifying non-negativity of polynomials are phrased as so-called Positivstellensätze in real algebra. For our purposes, we need certificates for certain spaces of invariant polynomials. Positivstellensätze in the invariant setting have been studied by different researchers such as Cordian Riener, Claus Scheiderer et al.~\cite{DR,HMR,Schei}. %TODO FIND LITERATURE)
    If a Positvistellensatz for degree $6$ copositive polynomials of the $2 \times 2 \times 2$ tensor which are invariant under the $48$ symmetries of the $3$-dimensional cube were known, %we could have used it for showing the copositivity of $H_3(x)$. 
%\Ivan{How about something like  "
it would have provided a more structural proof of the copositivity of $H_3(x)$, potentially applicable to $H_d(x)$ for $d>3$.  %"? }
	\end{rem}

\section{Probabilistic proof of the Hypercube inequality for an arbitrary $d$}\label{sect:probabilistic:proof}

\label{sect:probabilistic:proof} 

\newcommand{\Prob}{\mathsf{P}}
\newcommand{\Exp}{\mathsf{E}}
\newcommand{\Var}{\mathsf{Var}}
\newcommand{\Cov}{\mathsf{Cov}}
\newcommand{\tY}{\widetilde{Y}}
\newcommand{\tv}{\widetilde{v}}

In this section, for a vector $z \in \R^d$, we use the tilde notation to denote the vector $\tilde{z} = \begin{pmatrix} 1 \\ z \end{pmatrix} \in \R^{d+1}$.  We list vectors from $\{0,1\}^d$ in some order $v_0,\ldots,v_N$ with $N = 2^d-1$ and have a short-hand notation $x_i$ for $x_{v_i}$. We will use the notation $x_i$ with $i=0,\ldots,N$ and $x_v$ with $v \in \{0,1\}^d$ interchangeably. As the Hypercube inequality \eqref{eq:hypercube_ineq} is homogeneous  in the nonnegative variables $x_0,\ldots,x_N$, we assume without loss of generality that $\sum_{i=0}^N x_i = 1$.
In this setting, the Hypercube inequality can be formulated as

\[
\sum_{0 \le i_0 < \cdots i_d \le N}  x_{i_0} \cdots x_{i_d} \, \bigl| \det(\tv_{i_0},\ldots, \tv_{i_d}) \bigr|  \le \prod_{\substack{b \in \{0,1\} \\ i = 1,\ldots,d}} 
\left( \sum_{\substack{v \in \{0,1\}^d \colon \\ v_i = b}} x_v \right). 
\]
As $\det (\tv_{i_0},\ldots, \tv_{i_d})$ is a non-negative integer value, we notice that $\bigl| \det(\tv_{i_0},\ldots, \tv_{i_d}) \bigr| \le \det( \tv_{i_0},\ldots, \tv_{i_d} )^2$. Thus,   for showing the Hypercube inequality we verify a stronger inequality obtained by replacing the absolute value of the determinant on the left-hand side with the square of the determinant: 
\begin{equation} 
\label{eq:strong:ineq}
\sum_{0 \le i_0 < \cdots i_d \le N}  x_{i_0} \cdots x_{i_d} \, \det(\tv_{i_0},\ldots, \tv_{i_d})^2 \le \prod_{\substack{b \in \{0,1\} \\ i = 1,\ldots,d}} 
\left( \sum_{\substack{v \in \{0,1\}^d \colon \\ v_i = b}} x_v \right). 
\end{equation}

We use $\Prob, \Exp$ and $\Var$ to denote the probability, expected value and variance, respectively. 
The system of values $x_v$ with $v \in \{0,1\}^d$ is the probability distribution of a 0/1 random vector $Y = \begin{pmatrix} Y_1 \\ \vdots \\ Y_d \end{pmatrix}$. This means,  $\Prob( Y = v) = x_v$ for $v \in \{0,1\}^d$.  

We establish the probabilistic meaning of the right-hand side of \eqref{eq:strong:ineq}. The 0/1 random vector $Y_i$ has the variance 
\[
\Var(Y_i) = \Prob(Y_i=0) \Prob(Y_i=1) = \prod_{b \in \{0,1\}} \left( \sum_{\substack{v \in \{0,1\}^d \colon \\ v_i = b}} x_v \right),
\]
so that the right-hand of \eqref{eq:strong:ineq} is exactly $\prod_{i=1}^d \Var(Y_i)$. 

We now want to establish a probabilstic interpretation of the left-hand side of \eqref{eq:strong:ineq}. Consider the autocorrelation matrix of $\tY$,
\[
	\Exp( \tY \tY^\top) = \begin{pmatrix} 1 & \Exp(Y^\top) \\ \Exp(Y) & \Exp(Y Y^\top) \end{pmatrix},
\]
which is a symmetric positive semidefinite matrix that contains moments of the first and the second order of the random vector $Y$. Let us verify that the left-hand side of \eqref{eq:strong:ineq} coincides with 
 $\det \Exp(\tY \tY^\top)$. We have 
\begin{align*} 
	\det \Exp( \tY \tY^\top) & = \det \left(  \sum_{v \in \{0,1\}^d} x_v \, \tv \, \tv^\top \right)
\\ & = \det \underbrace{ \begin{pmatrix} x_0 \tv_0 & \cdots & x_N \tv_N \end{pmatrix}}_{(d+1) \times 2^d} \underbrace{\begin{pmatrix} \tv_0^\top \\ \vdots \\ \tv_N^\top \end{pmatrix} }_{2^d \times (d+1)}.
\end{align*} 
Applying the well-known Cauchy-Binet identity, we obtain 
\begin{align*} 
\det \Exp( \tY \tY^\top)  
& =  \sum_{0 \le i_0 < \cdots < i_d \le N}    \det(x_{i_0} \tv_{i_0},\ldots,  x_{i_d}\tv_{i_d}) \det(\tv_{i_0},\ldots, \tv_{i_d})
\\ & =  \sum_{0 \le i_0 < \cdots < i_d \le N}  x_{i_0} \cdots x_{i_d} \det(\tv_{i_0},\ldots, \tv_{i_d})^2. 
\end{align*} 
Consequently, \eqref{eq:strong:ineq} amounts to 
\[
	\det \Exp( \tY \tY^\top)  \le \prod_{i=1}^d \Var(Y_i). 
\]
It turns out that $\det \Exp( \tY \tY^\top) $ on the left-hand side is the determinant of the covariance matrix $\Cov(Y)$ of $Y$. 
Indeed, using Schur complement, we have 
\begin{align*} 
	\det \Exp(\tY \tY^\top ) & = \det \begin{pmatrix} 1 & \Exp(Y^\top) \\ \Exp(Y) & \Exp(Y Y^\top) \end{pmatrix}
	\\ & = 
	\det \begin{pmatrix} 1 & 0 \\
	\Exp(Y) & \operatorname{Id}
	\end{pmatrix} 
	\begin{pmatrix} 1 & 0 
	\\ 0 & \Exp(Y Y^\top)  - \Exp(Y) \Exp(Y^\top) \end{pmatrix} 
	\begin{pmatrix} 
	1 & \Exp(Y^\top)
	\\ 0 & \operatorname{Id} 
	\end{pmatrix} 
	\\ & = 
	\det 	\begin{pmatrix} 1 & 0 
	\\ 0 & \Exp(Y Y^\top)  - \Exp(Y) \Exp(Y^\top) \end{pmatrix} 
	\\ & = \det (\Exp(Y Y^\top)  - \Exp(Y) \Exp(Y^\top))
	\\ & = \det ( Y - \Exp(Y)) ( Y - \Exp(Y))^\top 
	\\ & = \det \Cov(Y). 
\end{align*} 
Hence, \eqref{eq:strong:ineq} becomes the inequality $\det \Cov(Y) \le \prod_{i=1}^d \Var(Y_i)$. This turns out to be  a well-known Hadamard inequality, since $\Cov(Y)$ is symmetric positive semidefinite and the variances $\Var(Y_i)$ are diagonal elements of $\Cov(Y)$. 

\begin{rem} 
	Recall that when  $|\det(\tv_{i_0},\ldots,\tv_{i_d})|$ is not equal to zero, it is the  normalized volume of the $0/1$ simplex $\conv(v_{i_0},\ldots,v_{i_d})$. In dimension $d=2$, all $0/1$ simplices have normalized volume $1$. Thus, 
	$|\det(\tv_{i_{0}},\ldots,\tv_{i_d})| = \det(\tv_{i_0},\ldots,\tv_{i_d})^2 \in \{0,1\}$ and therefore, in the case $d=2$, \eqref{eq:strong:ineq} is equivalent to the Hypercube inequality. However, starting from dimension $d=3$, there are $0/1$ simplices of normalized volume higher than one. Thus, for $d \ge 3$, \eqref{eq:strong:ineq} is strictly stronger than the Hypercube inequality. 
\end{rem} 

\begin{rem} 
	The proof in this section establishes a bridge between submodularity phenomena in probability and our submodularity-type inequality for mixed volumes of zonoids. The argument uses the  Hadamard inequality for symmetric positive semidefinite matrices. The Hadamard inequality is a direct consequence of a more general inequality of Hadamard-Fischer-Koteljanskii~\cite[Theorem~7.8.9]{HJ} that asserts, for a symmetric positive semidefinite matrix $Q$, the submodularity of the  function $f(I) = \log \det Q_I$, where $\det Q_I$ is the principal minor of $Q$ indexed by $I$.
\end{rem} 

%\begin{rem}
%	The  proof in this section sheds some light on the reason for the difference between submodularity phenomena in stochastic context in comparison to the submodularity in convex geometry. Very roughly, the difference is that where one has the absolute determinant in the convex setting, one usually has the squared determinant in the stochastic setting (WRITE PROPERLY).
	
%	 See also the Ruben, who verifies Artem's inequality for unimodular configurations (his proof looks a bit like the LLM proof) and refutation of Artem's inequality for bi-modular configurations. 
%\end{rem} 

%\input{bound_const}

%\input{impact}

%\input{story}


\begin{thebibliography}{10}

\bibitem[AAGJMR]{AAGJMR} Alonso-Gutiérrez, D., Artstein-Avidan, S., González Merino, B., Jiménez, C. H., and Villa, R. (2019). Rogers-Shephard and local Loomis-Whitney type inequalities. Math. Ann. 374(3-4), 1719–1771.

\bibitem[ABBC]{ABBC} Alonso-Gutiérrez, D., Bernués, J., Brazitikos, S., and Carbery, A. (2021). On affine invariant and local Loomis–Whitney type inequalities. J. Lond. Math. Soc. 103(4), 1377–1401.

\bibitem[AGM]{AGM} Alías, L. J., González Merino, B., and Marín Gimeno, B. (2025). On local Liakopoulos-Meyer type inequalities and their functional counterparts. arXiv:2512.02761. \url{https://arxiv.org/abs/2512.02761}.

\bibitem[AH]{AH} Adiprasito, K. and Huh, J. (2021). Log-concavity of matroid basis generating functions. Ann. Math. (2) 194, 819–868.

\bibitem[AS]{AS} Averkov, G. and Soprunov, I. (2025). An algebraic-combinatorial proof of a Bezout-type inequality for mixed volumes of three-dimensional zonoids. Discrete \& Computational Geometry, 1432-0444.

%\bibitem[B]{B} Bernstein, D. N. (1975). The number of roots of a system of equations. Funkcional. Anal. i Priložen. 9, 1–4.

%\bibitem[Be]{Be} Bézout, É. (1779). Théorie générale des équations algébriques. Paris: Ph.-D. Pierres.

%\bibitem[BG]{BG} Brazitikos, S. and Giannopoulos, A. (2018). Uniform cover inequalities for the volume of coordinate sections and projections of convex bodies. Adv. Geom. 18, 345–354.

%\bibitem[BGM]{BGM} Brandenberg, R. and González Merino, B. (2017). A complete 3-dimensional Blaschke-Santaló diagram. Math. Inequal. Appl. 20(2), 301–348.

%\bibitem[BGR1]{BGR1} Brandenberg, R., González Merino, B., and Runge, M. (2025). A complete system of inequalities for the diameter, in- and circumradius in the 3-dimensional Euclidean space. arXiv:2509.05028 [math.MG].

%\bibitem[BGR2]{BGR2} Brandenberg, R., González Merino, B., and Runge, M. (2026). Minimization of the inradius of convex bodies for prescribed diameter and circumradius in Minkowski spaces. arXiv:2606.15823 [math.MG].

\bibitem[BH]{BH} Brändén, P. and Huh, J. (2020). Lorentzian polynomials. Ann. Math. (2) 192, 821–891.

\bibitem[BBLM]{BBLM} Breiding, P., Bürgisser, P., Lerario, A., and Mathis, L. (2022). The zonoid algebra, generalized mixed volumes, and random determinants. \emph{Adv. Math.} 402, 108361.


%\bibitem[Bjo]{Bjo} Björner, A., Las Vergnas, M., Sturmfels, B., White, N., and Ziegler, G. M. (1999). Oriented Matroids. 2nd ed. Cambridge University Press.

\bibitem[Bl]{Bl} Blaschke, W. (1916). Eine Frage \"uber konvexe K\"orper. Jahresber. Deutsch. Math.-Verein. 25, 121–125.

%\bibitem[BR]{BR} Brandenberg, R. and Runge, M. (2023). Blaschke–Santaló diagrams for different diameter variants. Math. Inequal. Appl. (to appear).

\bibitem[BT]{BT} Bollobás, T. and Thomason, A. (1995). Projections of bodies and hereditary properties of hypergraphs. Bull. London Math. Soc. 27, 417–424.

%\bibitem[CC]{CC} Costa, M. H. and Cover, T. M. (1984). On the similarity of the entropy power inequality and the Brunn–Minkowski inequality. IEEE Trans. Inform. Theory 30, 837–839.

\bibitem[CLS]{CLS} Cox, D. A., Little, J. B., and Schenck, H. K. (2011). Toric Varieties. Graduate Studies in Mathematics, Vol. 124. American Mathematical Society, Providence, RI.

\bibitem[DCT]{DCT} Dembo, A., Cover, T. M., and Thomas, J. A. (1991). Information-theoretic inequalities. IEEE Trans. Inform. Theory 37, 1501–1518.

\bibitem[DR]{DR} Debus, S. and Riener, C. (2023). Reflection groups and cones of sums of squares. J. Symb. Comput. 117, Article 102243.

%\bibitem[Fe]{Fe} Fenchel, M. W. (1936). Généralisation du théoréme de Brunn et Minkowski concernant les corps convexes. C. R. Acad. Sci. Paris 203, 764–766.

\bibitem[Ful]{Ful} Fulton, W. (1993). Introduction to Toric Varieties. Annals of Mathematics Studies, Vol. 131. Princeton University Press, Princeton, NJ.


%\bibitem[FHL]{FHL} Ftouhi, I., Henrot, A., and Lamboley, J. (2025). Improved description of Blaschke–Santaló diagrams via numerical shape optimization. Appl. Math. Optim. 91, Article 55.

\bibitem[FHMNWZ]{FHMNZ} Fradelizi, M., Hubard, A., Manui, A., Ndiaye, C. S., Wang S., and Zvavitch, A. (2026). Volume and Projection Inequalities I: Zonoids and Courtade's Conjecture. 	arXiv:2608.12681 [math.MG]

%\bibitem[FLP]{FLP} Ftouhi, I. and Lamboley, J. (2022). Blaschke–Santaló diagrams and other shape optimization problems. Theses. hal-03252870.

\bibitem[FMMZ]{FMMZ} Fradelizi, M., Madiman, M., Meyer, M., and Zvavitch, A. (2024). On the volume of the Minkowski sum of zonoids. J. Funct. Anal. 286(3), Paper No. 110247, 41.

%\bibitem[FMZ]{FMZ} Fradelizi, M., Madiman, M., and Zvavitch, A. (2024). Sumset estimates in convex geometry. Int. Math. Res. Not. 15, 11426–11454.

%\bibitem[Fto]{Fto} Ftouhi, I. (2024). Numerical exploration of the range of shape functionals via Blaschke–Santaló diagrams. Preprint.

%\bibitem[G]{G} Gurvits, L. (2009). On multivariate Newton-like inequalities. In: Advances in Combinatorial Mathematics, pp. 61–78. Springer, Berlin.

\bibitem[GKZ]{GKZ} Gelfand, I. M., Kapranov, M. M., and Zelevinsky, A. V. (1994). \emph{Discriminants, resultants, and multidimensional determinants}. Birkhäuser, Boston.

\bibitem[H]{H} Heine, R. (1938). Der Wertvorrat der gemischten Inhalte von zwei, drei und vier ebenen Eibereichen. Math. Ann. 115, 115–129.

\bibitem[Ha]{Ha} Harris, J. (1992). Algebraic Geometry: A First Course. Graduate Texts in Mathematics, Vol. 133. Springer.

\bibitem[HC]{HC} Hernández Cifre, M. A. (2000). Is there a planar convex set with given width, diameter, and inradius? Amer. Math. Monthly 107(10), 893–900.

\bibitem[HJ]{HJ} Horn, R. A. and Johnson, C. R. (2012). Matrix Analysis. 2nd edn. Cambridge University Press, Cambridge.

%\bibitem[HCS]{HCS} Hernández Cifre, M. A., Salinas, G., and Segura, S. (2001). Complete systems of inequalities. J. Inequal. Pure Appl. Math. 2, 1–12.

\bibitem[HKL]{HKL} Hudelson, M., Klee, V., and Larman, D. (1996). Largest \(j\)-simplices in \(d\)-cubes: Some relatives of the Hadamard determinant problem. Linear Algebra Appl. 241–243, 519–598.

\bibitem[HMR]{HMR} Hubert, E., Metzlaff, T., Moustrou, P., and Riener, C. (2023). Optimization of trigonometric polynomials with crystallographic symmetry and spectral bounds for set avoiding graphs. arXiv:2303.09487. \url{https://arxiv.org/abs/2303.09487}.

%\bibitem[Kh]{Kh} Khovanskii, A. G. (1978). Newton polyhedra and the genus of complete intersections. Funkcional. Anal. i Priložen. 12(1), 51–61.

%\bibitem[KK]{KK} Kaveh, K. and Khovanskii, A. G. (2012). Algebraic equations and convex bodies. In: Perspectives in Analysis, Geometry, and Topology, pp. 263–282. Springer.

%\bibitem[Ku]{Ku} Kušnirenko, A. G. (1976). Newton polytopes and the number of solutions of a system of equations. Funkcional. Anal. i Priložen. 10(3), 82–83.

\bibitem[L]{L} Liakopoulos, D. M. (2019). Reverse Brascamp–Lieb inequality and the dual Bollobás–Thomason inequality. Arch. Math. 112(3), 293–304.

\bibitem[LW]{LW} Loomis, L. H. and Whitney, H. (1949). An inequality related to the isoperimetric inequality. Bull. Amer. Math. Soc. 55, 961–962.

\bibitem[M]{M} Meyer, M. (1988). A volume inequality concerning sections of convex sets. Bull. London Math. Soc. 20, 151–155.

\bibitem[Ma]{Ma} Marshall, M. (2008). Positive Polynomials and Sums of Squares. American Mathematical Society.

\bibitem[Md]{Md} Madiman, M. (2015). A survey of information theory and convex geometry. In: Information Theory and Applications Workshop, pp. 1–6.

\bibitem[MNZ]{MNZ} Manui, A., Ndiaye, C. S., and Zvavitch, A. (2024). On the volume of sums of anti-blocking bodies. arXiv:2409.14214. \url{https://arxiv.org/abs/2409.14214}.

%\bibitem[Mo]{Mo} Mondal, P. (2021). How many zeroes? Counting solutions of systems of polynomials via toric geometry at infinity. Springer.

%\bibitem[Mur]{Mur} Murota, K. (2003). Discrete Convex Analysis. SIAM.

\bibitem[NWZ]{NWZ} Neubauer, M. G., Watkins, W., and Zeitlin, J. (1997). Maximal \(j\)-simplices in the real \(d\)-dimensional unit cube. J. Comb. Theory Ser. A 80(1), 1–12.

%\bibitem[Ox]{Ox} Oxley, J. (2011). Matroid Theory. 2nd ed. Oxford University Press.

\bibitem[Re]{Re} Reznick, B. (2000). Some concrete aspects of Hilbert's 17th problem. In: Real Algebraic Geometry and Ordered Structures, Contemp. Math. 253, 251–272.

\bibitem[Sa]{Santalo61} Santaló, L. A. (1959/61). On complete systems of inequalities between elements of a plane convex figure. Math. Notae 17, 82–104.

\bibitem[Sch]{Sch} Schneider, R. (2014). Convex Bodies: The Brunn-Minkowski Theory. Encyclopedia of Mathematics and its Applications, Vol. 151, 2nd edn. Cambridge University Press, Cambridge.

\bibitem[Schei]{Schei} Scheiderer, C. (2012). A Positivstellensatz for projective real varieties. Manuscr. Math. 138(1-2), 73–88.

\bibitem[Sh]{Sh} Shephard, G. C. (1960). Inequalities between mixed volumes of convex sets. Mathematika 7, 125–138.

\bibitem[SY]{SY} Sanwine-Yager, J. R. (1989). The missing boundary of the Blaschke diagram. Amer. Math. Monthly 96, 233–237.

\bibitem[Sko]{Sko} Skorupinski, R. (2026). Zonoid volumes are not log-submodular. arXiv:2608.07702 [math.MG].

%\bibitem[SZ]{SZ} Soprunov, I. and Zvavitch, A. (2016). Bézout inequality for mixed volumes. Int. Math. Res. Not. 2016(23), 7230–7252.

\bibitem[V]{V} Vitale, R. A. (1991). Expected absolute random determinants and zonoids. Ann. Appl. Probab. 1(2), 293–300.

%\bibitem[Vo]{Vo} Vondrák, J. (2017). Submodular functions. Lecture notes, Math 233B, Stanford University. \url{https://theory.stanford.edu/~jvondrak/MATH233B-2017/lec14.pdf}

\bibitem[ZhMa]{ZhMa} Zhang, W. and Ma, H. (2026). Private communication. Lean formalization of the Hypercube inequality. Available at: \url{https://github.com/the-omega-institute/trureturing/blob/d4f8886a366b0c98bc7d62c9346d01b3040b69ab/D5/S0/FiniteGeometry/HypercubeInequality.lean}

\end{thebibliography}
\end{document}